\documentclass[10pt,reqno]{amsart}
\RequirePackage{amsthm,amsmath}
\RequirePackage[numbers]{natbib}
\RequirePackage[colorlinks,linkcolor=blue,citecolor=blue,urlcolor=blue]{hyperref}
\usepackage{amssymb}
\usepackage{geometry}
\usepackage{enumitem}

\numberwithin{equation}{section}

\theoremstyle{plain} 
\newtheorem{thm}{Theorem}[section]
\newtheorem{lem}[thm]{Lemma}

\newtheorem{pro}[thm]{Proposition}

\newtheorem{defn}[thm]{Definition}

\theoremstyle{remark}
\newtheorem{remark}{Remark}[section]
\newtheorem{fact}[remark]{Fact}

\newcommand{\w}{\mathbf w}

\newcommand{\compC}{\mathbb{C}}

\newcommand{\bigO}{{O}}

\newcommand{\realR}{\mathbb{R}}
\renewcommand{\u}{\mathbf{u}}

\newcommand{\x}{\mathbf{x}}

\allowdisplaybreaks

\DeclareMathOperator{\diag}{diag}

\DeclareMathOperator{\SC}{SC}
\DeclareMathOperator{\Span}{Span}
\DeclareMathOperator{\tr}{tr}
\DeclareMathOperator{\Var}{Var}

\begin{document}

 \begin{minipage}{0.85\textwidth}
 \vspace{2.5cm}
 \end{minipage}
\begin{center}
\large\bf
Eigenvector distribution of random matrices under critical finite-rank deformations
\end{center}
\renewcommand{\thefootnote}{\fnsymbol{footnote}}	
 \vspace{1cm}
\begin{center}
\begin{minipage}{0.28\textwidth}
\centering
Zhigang Bao\\
\footnotesize The University of Hong Kong\\
{\it zgbao@hku.hk}
\end{minipage}
\hfill
\begin{minipage}{0.28\textwidth}
\centering
Dong Wang\\
\footnotesize University of Chinese Academy of Sciences\\
{\it wangdong@wangd-math.xyz}
\end{minipage}
\hfill
\begin{minipage}{0.28\textwidth}
\centering
Yue Zhu\\
\footnotesize University of Chinese Academy of Sciences\\
{\it zhuyue242@mails.ucas.ac.cn}
\end{minipage}
\end{center}

\renewcommand{\thefootnote}{\fnsymbol{footnote}}	
\vspace{1cm}

\begin{center}
\begin{minipage}{0.8\textwidth}
\footnotesize
\setlength{\parindent}{2em}
\noindent{\bf Abstract.}
We investigate the eigenvector distribution at the soft edge for Gaussian random matrices with finite-rank deformations, in the critical regime of BBP transition. For finite-rank deformations of the GOE and GUE with critical spikes, we find that the squared overlap between a leading eigenvector and a spike, rescaled by \(N^{1/3}\), converges weakly to the negative reciprocal of the derivative of an Airy-Green function evaluated at the corresponding soft-edge root. For the  rank-one critically spiked Gaussian \(\beta\)-ensemble, \(\beta>0\), we obtain an analogous result involving an Airy-Green function. In both cases, the Airy-Green functions are generalizations of the one introduced by Bykhovskaya--Gorin--Sodin \cite{Bykhovskaya-Gorin-Sodin25}. The proofs are both based on an eigenvector--eigenvalue identity and a resolvent-differentiation mechanism. 
\end{minipage}
\end{center}

 \vspace{2mm}
 
 {\small
\footnotesize{\noindent\textit{Date}: \today}\\
 \footnotesize{\noindent\textit{Keywords}:}
 eigenvector overlap, finite-rank deformation, Airy-Green function, stochastic Airy operator
 
 \vspace{2mm}

 }

\thispagestyle{headings}

\tableofcontents

\section{Introduction}
Large random matrices with finite-rank deformation have been a widely used and well-studied model from both theoretical and applied perspectives in random matrix theory. A well-known result is the Baik-Ben Arous-P\'{e}ch\'{e} (BBP) phase transition \cite{Baik-Ben_Arous-Peche05, Peche05}, which depicts a phase transition of the distribution of the leading eigenvalues of spiked complex Wishart matrices and spiked Gaussian unitary ensemble (GUE). In particular, the distribution in the critical regime, as an interpolation between the Tracy-Widom distribution and the Gaussian distribution, was discovered therein. The critical distribution in the real case, and a further extension to more general $\beta$-ensembles, was later given in \cite{Bloemendal-Virag11,Bloemendal-Virag11a} as the smallest eigenvalues of the stochastic Airy operator with a certain boundary condition. Very recently, the critical eigenvalue was alternatively represented as a root of the so-called Airy-Green function in \cite{Bykhovskaya-Gorin-Sodin25}, for the real case. 

Apart from the mentioned results on the critical eigenvalue distributions for the Wishart and Gaussian ensembles, there have been a vast number of results on the subcritical and supercritical regimes, and also extensions to more generally distributed random matrices with low-rank deformations in all regimes. Without being comprehensive, we refer to \cite{Bai-Yao08, Baik-Wang10a, Baik-Wang10, Benaych_Georges-Guionnet-Maida11, Capitaine-Donati_Martin-Feral09, Feral-Peche07, Knowles-Yin13, Knowles-Yin14} and the references therein.

Along with the BBP transition for the leading eigenvalues, there is a corresponding transition for the leading eigenvectors. In particular, in the supercritical regime, the top eigenvectors have a non-negligible overlap with the spike eigenvectors, while in the subcritical regime, they are nearly orthogonal; see \cite{Benaych_Georges-Nadakuditi11} for instance. Results on finer scales including precise fluctuation results have also been obtained in the subcritical and supercritical regimes under rather general distributional assumptions of the matrix entries; see \cite{Bloemendal-Knowles-Yau-Yin16, Bao-Ding-Wang18, Bao-Ding-Wang-Wang20, Capitaine-Donati_Martin18} for instance. We also refer to \cite{Benaych_Georges-Nadakuditi12, Capitaine18, Paul07, Mo11a, Ding20}, etc., for related studies.

In contrast, as an analogue of the main BBP result for the critical eigenvalue distribution in \cite{Baik-Ben_Arous-Peche05, Peche05}, the eigenvector distribution in the critical regime has been much less studied. On the fluctuation level, the only available result was obtained in \cite{Bao-Wang20}, for the spiked GUE model.  In \cite{Bao-Wang20}, using an eigenvector-eigenvalue identity, together with the determinantal structure of the GUE minor process with an external source, the eigenvector distribution can be represented by an extended Airy process. However, such a representation cannot be extended to spiked GOE due to the lack of an explicit distribution of the GOE minor process with an external source, not to mention the spiked Gaussian $\beta$-ensemble for general $\beta>0$. In this work, we aim to introduce a unified representation for the leading eigenvector of the Gaussian orthogonal/unitary ensembles with finite-rank deformation by using the Airy-Green function introduced recently in \cite{Bykhovskaya-Gorin-Sodin25}, and extend the result to the Gaussian $\beta$-ensemble for general $\beta$, in the rank one case.

Below, we first introduce some preliminaries and notation in
Section~\ref{s.preliminaries}, and then state our main results in
Section~\ref{s.main results}. We also record in
Proposition~\ref{pro:existence_rank_r} the almost-sure simplicity of
the higher-rank Airy$_1$ and Airy$_2$ point processes, as an auxiliary
consequence of the arguments in
Sections~\ref{subsec:proof_mult_rank_GOE} and
\ref{subsec:proof_GUE}. 

\subsection{Preliminaries and notations} \label{s.preliminaries}

In this subsection we introduce our matrix models and collect some necessary inputs used throughout the paper. We first recall the classical Gaussian ensembles and the Dumitriu--Edelman tridiagonal Gaussian \(\beta\)-ensemble, and their finite-rank deformations. Further, we state some  results regarding the stochastic Airy operator and the Airy-Green function. 

Most of the results stated in this subsection are known in literature. The proofs of the new properties of the Airy-Green functions stated in this subsection, in particular Propositions \ref{prop:beta_airy_green}, \ref{pro:multispike_airy_green}, and Lemma~\ref{lem:beta_airy_green_infty}, are deferred to Section~\ref{sec:supplemental}.

\subsubsection{Gaussian orthogonal ensemble, Gaussian unitary ensemble and Gaussian \texorpdfstring{$\beta$}{beta}-ensemble}

Let \(H_N=(h_{ij})_{1\le i,j\le N}\) be an \(N\times N\) random matrix. We say that \(H_N\) is a \emph{Gaussian Orthogonal Ensemble} (GOE) matrix if \(H_N\) is real symmetric,
\begin{align} \label{eq:defn_GOE}
h_{kk}\sim {}& \mathcal N\!\left(0,\frac{2}{N}\right),\quad 1 \leq k \leq N, & h_{ij}\sim {}& \mathcal N\!\left(0,\frac{1}{N}\right),\quad 1 \leq i<j \leq N,
\end{align}
and all diagonal and upper-triangular entries are independent. We say that \(H_N\) is a \emph{Gaussian Unitary Ensemble} (GUE) matrix if \(H_N\) is complex Hermitian,
\begin{align} \label{eq:defn_GUE}
  h_{kk}\sim {}& \mathcal N\!\left(0,\frac{1}{N}\right),\quad 1 \leq k \leq N, & h_{ij}\sim {}& \mathcal N\!\left(0,\frac{1}{2N}\right) +i\,\mathcal N\!\left(0,\frac{1}{2N}\right),\quad 1 \leq i<j \leq N,
\end{align}
and all diagonal and upper-triangular entries are independent. In either case, we write
\begin{align} \label{matrix eigva}
  \mu_1\ge \mu_2\ge \cdots \ge \mu_N
\end{align}
for the eigenvalues of \(H_N\) in descending order.

For the above GOE and GUE matrices, the joint probability density function (j.p.d.f.)  of \(\{\mu_j\}_{j=1}^N\) is
\begin{equation}\label{betaEpdf}
  \frac{1}{Z_{N,\beta}}
  \prod_{1\le i<j\le N}|\mu_i-\mu_j|^\beta
  \exp\!\left(-\frac{\beta N}{4}\sum_{i=1}^N\mu_i^2\right),
\end{equation}
with \(\beta=1\) in the GOE case and \(\beta=2\) in the GUE case. 

For general \(\beta>0\), the $N$-dimensional \emph{Gaussian $\beta$-ensemble} (G$\beta$E) is realized by the Dumitriu--Edelman tridiagonal random matrix \cite{Dumitriu-Edelman02}
\begin{equation}\label{eq:defn_H^beta}
H_N^{(\beta)}
=
\frac{1}{\sqrt{\beta N}}
\begin{pmatrix}
a_1^{(\beta)} & b_1^{(\beta)} & & & \\
b_1^{(\beta)} & a_2^{(\beta)} & b_2^{(\beta)} & & \\
& b_2^{(\beta)} & a_3^{(\beta)} & \ddots & \\
& & \ddots & \ddots & b_{N-1}^{(\beta)} \\
& & & b_{N-1}^{(\beta)} & a_N^{(\beta)}
\end{pmatrix},
\end{equation}
where the entries are independent and satisfy
\begin{align*} % \label{eq:defn_H^beta_entries}
  a_j^{(\beta)}\sim {}& \mathcal N(0,2), \quad 1\le j\le N,  &  b_k^{(\beta)}\sim {}& \chi_{\beta(N-k)}, \quad 1\le k\le N-1.
\end{align*}
The j.p.d.f. of the eigenvalues of $H^{(\beta)}_N$ for general $\beta > 0$ is given by \eqref{betaEpdf}; see \cite{Dumitriu-Edelman02}. We denote the eigenvalues of \(H_N^{(\beta)}\) in descending order by
\begin{equation} \label{eq:defn_mu^beta_k}
\mu_1^{(\beta)}\ge \mu_2^{(\beta)}\ge \cdots \ge \mu_N^{(\beta)}.
\end{equation}
For \(\beta=1,2\), the distribution of \(\{\mu^{(\beta)}_j\}_{j=1}^N\) coincides with the GOE/GUE eigenvalue law.

The empirical spectral measure of a GOE/GUE matrix \(H_N\), and more generally that of \(H_N^{(\beta)}\), converges almost surely to the semicircle law
\[
d\mu_{\SC}(x)=\frac{\sqrt{4-x^2}}{2\pi}\,dx,\qquad x\in[-2,2].
\]
The Stieltjes transform of semicircle law is
\begin{align}
    G_{\SC}(z)
=\int_{\mathbb R}\frac{1}{z-x}\,d\mu_{\SC}(x)=
\frac{z-\sqrt{z^2-4}}{2},
\qquad z\in\mathbb C\setminus[-2,2],\label{Stieltjes transform of sc}
\end{align}
where the branch is chosen so that \(\sqrt{z^2-4}\sim z\) as \(|z|\to\infty\). Moreover, the top eigenvalue converges almost surely to \(2\); see
\cite{Bourgade-Erdos-Yau14} for instance. 

\subsubsection{Spiked deformations} %\label{subsubsec:spiked_deform}

Throughout the paper, we equip  $\mathbb{R}^N$ with the standard Euclidean metric and \(\mathbb C^N\) with the standard Hermitian metric. Let
\[e_i=(0,\dots,\underset{i-th}{1},0,\dots,0)^\top, \quad 1 \leq i \leq N, \]
be the canonical basis vectors of \(\mathbb R^N\) or \(\mathbb C^N\), according to the ensemble under consideration, and \(\langle\cdot,\cdot\rangle\) denotes the standard Euclidean or Hermitian inner product.

We consider additive finite-rank deformations of \(H_N\) or \(H^{(\beta)}_N\) in the form of
\begin{align}\label{eq:spiked_model}
  Y= {}& X+P, && \text{where}  & P = {}& \sum_{i=1}^r d_i \w_i\w_i^\ast .
\end{align}
When $X$ is either a GOE  or GUE matrix \(H_N\), $r$ can be any fixed positive integer and the vectors \(\w_1,\dots,\w_r\) are deterministic orthonormal vectors in \(\mathbb R^N\) in the real case and in \(\mathbb C^N\) in the complex case.  When $X$ is \(H^{(\beta)}_N\) and we take \(r=1\) and \(\w_1=e_1\). Denote the eigenvalues of \(Y\) by
\begin{equation*} % \label{eq:eigen_Y}
\lambda_1\ge \lambda_2\ge \cdots \ge \lambda_N,
\end{equation*}
and let \(\x_k\) be a unit eigenvector associated with \(\lambda_k\). Whenever the spectrum of \(Y\) is simple, \(\x_k\) is determined by \(Y\) and \(\lambda_k\) up to a sign in the real case and up to a phase in the complex case. This simplicity holds almost surely in the cases considered below. Hence the overlaps
\begin{equation} \label{eq:overlap}
  \lvert \langle \x_k,\w_j\rangle\rvert^2,\qquad j=1,\dots,r,
\end{equation}
are unambiguously defined. For simplicity, we will call the quantities in (\ref{eq:overlap}) the eigenvector overlaps in the sequel. 

\begin{remark}
  Since  GOE/GUE is orthogonally/unitarily invariant,  the joint law of the eigenvector overlaps is unchanged if we replace $\w_j$'s by $e_j$'s. 
\end{remark}

\begin{remark}
When \(\beta=1,2\) the eigenvector overlap law for \(H_N^{(\beta)}+d_1e_1e_1^\ast\) is the same as that for \( H_N+d_1e_1e_1^\ast\), where \(H_N\) is the GOE matrix for \(\beta=1\) and the GUE matrix for \(\beta=2\). Indeed, the Dumitriu--Edelman/Lanczos tridiagonalization is implemented by an orthogonal conjugation in the GOE case and by a unitary conjugation in the GUE case. In both cases the conjugating matrix can be chosen to fix \(e_1\), and therefore it leaves \( e_1e_1^\ast\) unchanged. The eigenvectors are conjugated, but
\[
  \lvert \langle \x_k,e_1\rangle\rvert
\]
is preserved. Hence the rank-one overlap law of the invariant GOE/GUE model coincides with that of the tridiagonal model \(H_N^{(\beta)}+d_1e_1e_1^\ast\) for \(\beta=1,2\).
\end{remark}

\subsubsection{Scalar and multivariate stochastic Airy operators}

Let \(b=(b_x)_{x\ge0}\) be a standard real Brownian motion. For \(\beta>0\), the scalar stochastic Airy operator is formally defined as 
\begin{equation}\label{eq:SAO}
\mathcal H_\beta
=
-\frac{d^2}{dx^2}+x+\frac{2}{\sqrt{\beta}}\,b_x',
\end{equation}
on \((0,\infty)\), where \(b_x'\) is real white noise, understood as the distributional derivative of \(b_x\). For \(\theta\in\mathbb R\cup\{\infty\}\), we impose the Robin boundary condition
\[
f'(0)=\theta f(0),
\]
with the convention that \(\theta=\infty\) means the Dirichlet boundary condition \(f(0)=0\).

The construction of \(\mathcal H_\beta\) with Dirichlet boundary
condition was established in \cite{Ramirez-Rider-Virag11}. The Robin
boundary condition
\[
f'(0)=\theta f(0)
\]
is treated in \cite[Section~2]{Bloemendal-Virag11}. The existence and
discreteness of the spectrum follow from
\cite[Theorem~2.2 and Corollary~2.3]{Gaudreau_Lamarre-Ghosal-Li-Liao23}, while its simplicity
follows from \cite[Lemma~2.7]{Bloemendal-Virag11}.

\begin{pro} % \label{pro:existence_rank_1}
Fix \(\beta>0\). For every
\(\theta\in\mathbb R\cup\{\infty\}\), the operator
\(\mathcal H_\beta\) with boundary condition
\(f'(0)=\theta f(0)\), interpreted as \(f(0)=0\) when
\(\theta=\infty\), has, almost surely, a purely discrete, simple
spectrum bounded from below. We enumerate its eigenvalues in strictly
increasing order as
\[
\Lambda_{\beta,1}^{(\theta)}
<
\Lambda_{\beta,2}^{(\theta)}
<
\cdots,
\qquad
\Lambda_{\beta,k}^{(\theta)}\longrightarrow+\infty.
\]
\end{pro}

The following convergence statements are essentially from \cite[Theorem~1.5]{Bloemendal-Virag11} and \cite[Theorem~1.1]{Ramirez-Rider-Virag11}. In particular, the unperturbed Dirichlet case, i.e., \(\theta=\infty\), is equivalent to the soft-edge limit in \cite[Theorem~1.1]{Ramirez-Rider-Virag11}.

\begin{pro}\label{pro:convergence_rank_1}
Fix \(\beta>0\), \(\theta\in\mathbb R\), and \(k\ge1\). Let \(\mu_1^{(\beta)}\ge \cdots \ge \mu_N^{(\beta)}\) be the eigenvalues of \(H_N^{(\beta)}\), and let \(\lambda_1\ge\cdots\ge\lambda_N\) be the eigenvalues of
\[
  H_N^{(\beta)}+d e_1e_1^\top,
  \qquad
  d=1-\theta N^{-1/3}.
\]
Then
\begin{equation*} % \label{eq:AP_rank_0}
\bigl(
N^{2/3}(2-\mu_1^{(\beta)}),\dots,N^{2/3}(2-\mu_k^{(\beta)})
\bigr)
\Longrightarrow
\bigl(
\Lambda_{\beta,1}^{(\infty)},\dots,\Lambda_{\beta,k}^{(\infty)}
\bigr),
\end{equation*}
and
\begin{equation*} % \label{eq:AP_rank_1}
\bigl(
N^{2/3}(2-\lambda_1),\dots,N^{2/3}(2-\lambda_k)
\bigr)
\Longrightarrow
\bigl(
\Lambda_{\beta,1}^{(\theta)},\dots,\Lambda_{\beta,k}^{(\theta)}
\bigr).
\end{equation*}
\end{pro}

In the case of spiked GOE/GUE with rank \(m > 1\) deformation, one needs the multivariate stochastic Airy operator introduced by Bloemendal and Vir\'ag \cite[Section~1]{Bloemendal-Virag11a}. This operator describes the soft-edge limit of the finite-rank invariant model after keeping track of the whole \(m\)-dimensional spiked subspace. At the discrete level, Bloemendal and Vir\'ag introduce a
canonical $(2m+1)$-diagonal band form. After grouping the
coordinates into blocks of size $m$, this form may equivalently
be viewed as a block-tridiagonal matrix.

Let \(\mathbb F=\mathbb R\) for \(\beta=1\) and
\(\mathbb F=\mathbb C\) for \(\beta=2\). A standard matrix
Brownian motion \(B=(B_x)_{x\geq 0}\) in the self-adjoint
\(m\times m\) matrices over \(\mathbb F\) is normalized so that,
for \(0\leq y\leq x\), \(B_x-B_y\) has the law of
\(\sqrt{m(x-y)}\) times an \(m\times m\) GOE matrix when
\(\beta=1\), and \(\sqrt{m(x-y)}\) times an \(m\times m\) GUE
matrix when \(\beta=2\), where the GOE/GUE matrices have the
normalization specified above.

Equivalently, for \(\beta=1\), the diagonal entries are independent
Brownian motions with variance parameter \(2\), and the
upper-triangular off-diagonal entries are independent standard real
Brownian motions. For \(\beta=2\), the diagonal entries are
independent standard real Brownian motions, and the upper-triangular
off-diagonal entries are independent standard complex Brownian
motions satisfying
\[
  \mathbb E
  \left\lvert B_{ij}(x)-B_{ij}(y)\right\rvert^2
  =x-y.
\]
All these processes are mutually independent, subject to
self-adjoint symmetry.

\begin{equation*} %\label{eq:multivariate_SAO}
\pmb{\mathcal H}_{\beta,m}
=
-\frac{d^2}{dx^2}
+mx\,I_m
+\sqrt{2}\,B_x',
\end{equation*}
acting on \(L^2(\mathbb R_+;\mathbb F^m)\). We impose the self-adjoint boundary condition
\begin{equation*} % \label{eq:Robin_matrix}
f'(0)=Wf(0),
\end{equation*}
where \(W\) is real symmetric for \(\beta=1\) and Hermitian for
\(\beta=2\). Since the law of \(B_x\) is invariant under fixed
orthogonal/unitary conjugation, the law of the spectrum depends on
\(W\) only through its eigenvalues. Hence, for distributional purposes,
we may assume that
\[
W=\operatorname{diag}(\theta_1,\dots,\theta_m).
\]
Writing
\[
\Theta=(\theta_1,\dots,\theta_m)\in\mathbb R^m,
\]
the boundary condition becomes
\begin{equation}\label{eq:boundary_condition_diagonal}
f_i'(0)=\theta_i f_i(0),
\qquad i=1,\dots,m.
\end{equation}
This is precisely the Robin boundary convention used in
\cite[Section~1, Equations~(1.1)--(1.2)]
{Bloemendal-Virag11a}. We record the spectral properties of
\(\pmb{\mathcal H}_{\beta,m}\) that will be used below.

\begin{pro}\label{pro:existence_rank_r}
Let \(m\in\mathbb N\) and let
\(\Theta=(\theta_1,\dots,\theta_m)\in\mathbb R^m\) be fixed.
For \(\beta=1,2\), the operator
\(\pmb{\mathcal H}_{\beta,m}\) with boundary condition
\eqref{eq:boundary_condition_diagonal} has, almost surely, a purely
discrete, simple spectrum bounded from below. We write its eigenvalues
in strictly increasing order as
\begin{equation*} %\label{eq:eigenvalues_matrix_SAO}
\Lambda_{\beta,1}^{\Theta}
<
\Lambda_{\beta,2}^{\Theta}
<
\cdots,
\qquad
\Lambda_{\beta,k}^{\Theta}\longrightarrow+\infty.
\end{equation*}
\end{pro}

\begin{remark}
The lower boundedness of the spectrum in
Proposition~\ref{pro:existence_rank_r} follows from the quadratic-form
estimate \cite[Equation~(3.16)]{Bloemendal-Virag11a}, while its
existence and discreteness follow from
\cite[Proposition~3.7 and Remark~3.8]{Bloemendal-Virag11a}.
In the scalar case $m=1$, simplicity follows from
\cite[Lemma~2.7]{Bloemendal-Virag11}. For $m\ge2$, simplicity is
established as part of our proofs in
Sections~\ref{subsec:proof_mult_rank_GOE} and
\ref{subsec:proof_GUE} for $\beta=1$ and $\beta=2$, respectively.
We note that the simplicity conclusion can also be deduced
from the very recent multilevel construction of Landon and Stojimirovic
\cite[Equations~(1.7)--(1.8), Proposition~1.1, and
Theorem~1.2]{Landon-Stojimirovic26}, in which the successive limiting
point processes are strictly interlacing.
\end{remark}

When \(m=1\) and \(\Theta=(\theta)\), the operator
\(\pmb{\mathcal H}_{\beta,1}\) coincides with
\(\mathcal H_\beta\) defined in \eqref{eq:SAO} with Robin boundary
parameter \(\theta\), for \(\beta=1,2\). The next proposition is a
rescaled special case of
\cite[Theorem~1.2]{Bloemendal-Virag11a} for \(\beta=1,2\).

\begin{pro}\label{pro:convergence_rank_r}
Let \(H_N\) be a GOE matrix when \(\beta=1\) and a GUE matrix when \(\beta=2\). Let
\begin{align*}
  P = \sum_{i=1}^md_ie_ie_i^\ast = \diag(d_1,\dots,d_m,0,\dots,0), \qquad d_i=1-\theta_iN^{-1/3}, \qquad i= {} 1,\dots,m,
\end{align*}
with $\theta_i \in \realR$. Write
\[
\lambda_1\ge\cdots\ge\lambda_N
\]
for the eigenvalues of \(H_N+P\). Then, for each fixed \(k\ge1\),
\[
\bigl(
N^{2/3}(2-\lambda_1),\dots,N^{2/3}(2-\lambda_k)
\bigr)
\Longrightarrow
\bigl(
\Lambda_{\beta,1}^{(\theta_1,\dots,\theta_m)},\dots,
\Lambda_{\beta,k}^{(\theta_1,\dots,\theta_m)}
\bigr).
\]
\end{pro}

\subsubsection{Airy-Green functions}

The Airy\(_\beta\) point process \(\{a_{\beta,k}\}_{k=1}^\infty\) is the negative Dirichlet spectrum of the scalar stochastic Airy operator:
\begin{equation}\label{airy process}
a_{\beta,k}:=-\Lambda_{\beta,k}^{(\infty)},
\qquad k\ge1,
\end{equation}
arranged in decreasing order.

For \(\beta\in\{1,2\}\) and a vector \(\Theta=(\theta_1,\dots,\theta_q)\), \(q\ge1\), we write
\begin{equation*} % \label{eq:defn_a_beta_j^Theta}
a_{\beta,j}^{\Theta}
:=
-\Lambda_{\beta,j}^{(\theta_1,\dots,\theta_q)},
\qquad j\ge1,
\end{equation*}
for the negative spectrum of the \(q\)-component multivariate stochastic Airy operator with boundary parameters \(\theta_1,\dots,\theta_q\), arranged in decreasing order. We refer to this as the higher-rank Airy$_\beta$ point process depending on \(\Theta\) for $\beta=1,2$ \cite[Remark 4.10]{Bykhovskaya-Gorin-Sodin25}. For notational convention, we also allow \(\Theta=()\), and then let \(a_{\beta,j}^{()}=a_{\beta,j}\) defined in \eqref{airy process}.

We now recall the Airy-Green function defined in \cite{Bykhovskaya-Gorin-Sodin25}. For \(\beta=1\), the construction associated with the Airy$_1$ point process is given in \cite[Theorem~8.1]{Bykhovskaya-Gorin-Sodin25}; for general \(\beta>0\), the corresponding extension associated with the Airy$_\beta$ point process is indicated in \cite[Remark~4.5, Equation~(4.5)]{Bykhovskaya-Gorin-Sodin25} and we state its definition in the following proposition.

\begin{pro}\label{prop:beta_airy_green}
Let \(\{a_{\beta,j}\}_{j\ge1}\) be the Airy\(_\beta\) point process, arranged in decreasing order, and let \(\{\xi_{\beta,j}\}_{j\ge1}\) be independent and identically distributed (i.i.d.) random variables, independent of \(\{a_{\beta,j}\}_{j\ge1}\), with \(\xi_{\beta,j}^2\sim\chi_\beta^2\). Then, almost surely, the Airy-Green function for general $\beta>0$
\begin{equation}\label{eq:airy_green_beta_limit}
    \mathcal G^{(\beta)}(w)
    :=
    \lim_{x\to-\infty}
    \left[
      \sum_{j:\,a_{\beta,j}>x}
      \frac{\beta^{-1}\xi_{\beta,j}^2}{w-a_{\beta,j}}
      -\frac{2}{\pi}\sqrt{-x}
    \right]
  \end{equation}
  exists for every \(w\in\mathbb C\setminus\{a_{\beta,j}:j\ge1\}\), uniformly on compact subsets avoiding the poles.
\end{pro}

In our paper, we also define the Airy-Green function associated with the higher-rank Airy$_\beta$ process associated with the boundary parameters \(\Theta=(\theta_1,\dots,\theta_q)\) with $\beta = 1, 2$. 
\begin{pro}\label{pro:multispike_airy_green}
  Let \(\beta\in\{1,2\}\) and \(\{\xi_{\beta,j}\}_{j\ge1}\) be i.i.d.\ random variables, independent of \(\{a_{\beta,j}^{\Theta}\}_{j\ge1}\), with \(\xi_{\beta,j}^2\sim\chi_\beta^2\). Define the Airy-Green function
  \begin{equation}\label{eq:defn_Airy_Green_multiple}
    \mathcal G^{(\beta)}_{\Theta}(w)
    :=
    \lim_{x\to-\infty}
    \left[
      \sum_{j:\,a_{\beta,j}^{\Theta}>x}
      \frac{\beta^{-1}\xi_{\beta,j}^2}{w-a_{\beta,j}^{\Theta}}
      -\frac{2}{\pi}\sqrt{-x}
    \right].
  \end{equation}
  For every fixed finite vector \(\Theta\), the limit in \eqref{eq:defn_Airy_Green_multiple} exists almost surely for every
  \[
    w\in
    \mathbb C\setminus
    \bigl\{
    a_{\beta,j}^{\Theta}:j\ge1
    \bigr\},
  \]
  uniformly on compact subsets avoiding the poles.
\end{pro}

For notational convention, we also allow \(\Theta=()\), and then let \(\mathcal{G}_{()}^{(\beta)}(w)=\mathcal{G}^{(\beta)}(w)\) defined in \eqref{eq:airy_green_beta_limit}.

For later use, we record the following asymptotic behavior of the
Airy Green functions at infinity in the right half plane.

\begin{lem} \label{lem:beta_airy_green_infty}
  Let \(\mathcal G_\star\) denote either \(\mathcal G^{(\beta)}\) defined in \eqref{eq:airy_green_beta_limit}, or \(\mathcal G^{(\beta)}_{\Theta}\) defined in \eqref{eq:defn_Airy_Green_multiple}.
  \begin{enumerate}
  \item \label{enu:lem:beta_airy_green_infty:1} 
    Almost surely,
    \[
      \lim_{\substack{|w|\to\infty\\ \Re w\ge0}}
      \bigl|\mathcal G_\star(w)+\sqrt w\bigr|=0,
    \]
    where the square root is the principal branch, positive on \((0,\infty)\). Consequently,
    \[
      \mathcal G_\star(x)\longrightarrow-\infty,
      \qquad x\to+\infty.
    \]
  \item \label{enu:lem:beta_airy_green_infty:2}  
    If, in addition, \(\beta\in\{1,2\}\), then, along the positive real axis,
    \begin{equation*} % \label{eq:key_formula_only_beta_12}
      \begin{pmatrix}
        x^{1/4}\bigl(\mathcal G_\star(x)+\sqrt{x}\bigr)\\[0.4em]
        x^{1/4}\left(
        -\dfrac{1}{\mathcal G_\star'(x)}-2\sqrt{x}
        \right)
      \end{pmatrix}
      \Longrightarrow
      \mathcal N\!\left(
        \begin{pmatrix}0\\0\end{pmatrix},
        \frac1\beta
        \begin{pmatrix}
          1&-1\\
          -1&2
        \end{pmatrix}
      \right).
    \end{equation*}
    Moreover, almost surely,
    \[
      \mathcal G_\star'(x)
      =
      -\frac1{2\sqrt{x}}+o(x^{-1}),
      \qquad
      \mathcal G_\star''(x)
      =
      \frac1{4x^{3/2}}+o(x^{-2}).
    \]
  \end{enumerate}
\end{lem}
We note that Proposition \ref{prop:beta_airy_green} is an extension of \cite[Theorem~8.1]{Bykhovskaya-Gorin-Sodin25} from $\beta=1$ to $\beta > 0$, and the $\mathcal{G}_{\star} = \mathcal{G}^{(1)}$ case of Part \ref{enu:lem:beta_airy_green_infty:1} of Lemma \ref{lem:beta_airy_green_infty} is obtained in \cite[Proposition~8.3]{Bykhovskaya-Gorin-Sodin25}.

\subsection{Main results} \label{s.main results}

On the fluctuation level,  the critical BBP transition parameters for the spiked GOE/GUE are
\begin{equation} \label{eq:defn_d_i}
  d_i=1-\theta_iN^{-1/3}, \qquad\theta_i\in \mathbb R,\qquad i=1,\dots,r,
\end{equation}
and the corresponding spiked model is denoted by 
\begin{align} \label{eq:defn_H^D_N}
H_N^{(D)}= {}& H_N+\sum_{i=1}^r d_i \w_i \w_i^\top  & \text{or} && H_N^{(D)}= {}& H_N+\sum_{i=1}^r d_i \w_i \w_i^\ast,
\end{align}
as described in \eqref{eq:spiked_model}. In the rank-one tridiagonal G\(\beta\)E setting, we consider
\begin{equation} \label{eq:defn_H^beta+dee}
H_N^{(\beta)}+d e_1e_1^\top,
\qquad
d=1-\theta N^{-1/3}, \qquad \theta\in \mathbb{R}. 
\end{equation}

We now state our main results regarding the eigenvector overlap distribution in the critical BBP transition regime. The first result concerns general rank $r$ deformation for GOE and GUE respectively. The second result gives the extension to the G\(\beta\)E in the rank-one case. In all cases, the  overlap scale is \(N^{-1/3}\). After a normalization, the limiting law is expressed in the negative reciprocal of the derivative of the corresponding Airy-Green function, evaluated at the appropriate soft-edge root.

\begin{thm}[GOE and GUE with finite-rank critical deformation]\label{finite_spikes_GOE}
Let \(H_N\) be the GOE (resp.\ GUE) matrix of dimension \(N\) specified by \eqref{eq:defn_GOE} (resp.~\eqref{eq:defn_GUE}). Fix \(r\ge1\), and let $d_1, \dotsc, d_r$ be specified in \eqref{eq:defn_d_i}, $\w_1, \dotsc, \w_r$ be deterministic orthonormal vectors, and the spiked model $H^{(D)}_N$ be the $N \times N$ random matrix defined in \eqref{eq:defn_H^D_N}.

Let
\[
\lambda_1^{(D)}\ge \lambda_2^{(D)}\ge \cdots \ge \lambda_N^{(D)}
\]
be the eigenvalues of \(H_N^{(D)}\), and let \(\x_k^{(D)}\) be a unit eigenvector associated with \(\lambda_k^{(D)}\). Fix \(i\in\{1,\dots,r\}\) and let \(k\ge1\) be independent of \(N\). Let $\beta = 1$ for the GOE case and $\beta = 2$ for the GUE case. Put
\begin{equation}\label{eq:Theta_i}
  \Theta^{(i)}
  =
  (\theta_1,\dots,\theta_{i-1},\theta_{i+1},\dots,\theta_r),
  \qquad
  \mathcal G^{(\beta)}_{(r-1),i^c}:=\mathcal G^{(\beta)}_{\Theta^{(i)}},
\end{equation}
where \(\mathcal G^{(\beta)}_{\Theta^{(i)}}\) is the Airy-Green function from Proposition~\ref{pro:multispike_airy_green}. Let \(s_{\beta, k}^{(i)}\) be the unique solution in
\[
\bigl(
a_{\beta,k}^{(\Theta^{(i)})},
a_{\beta,k-1}^{(\Theta^{(i)})}
\bigr), \qquad a_{\beta,0}^{(\Theta^{(i)})}:= +\infty,
\]
of
\begin{equation} \label{eq:roots_beta_12}
  \mathcal G^{(\beta)}_{(r-1),i^c}(x)=\theta_i.
\end{equation}
Then, as \(N\rightarrow\infty\), we have the joint convergence in distribution
\begin{align}\label{GOE_Airy Green fun}
 N^{2/3}\bigl(\lambda_k^{(D)}-2\bigr) \Longrightarrow {}& s_{\beta,k}^{(i)},  & N^{1/3}\bigl|\langle \w_i,\x_k^{(D)}\rangle\bigr|^2
\Longrightarrow {}&
-\frac{1}{(\mathcal G^{(\beta)}_{(r-1),i^c})'\!\bigl(s_{\beta, k}^{(i)}\bigr)}.
\end{align}
Moreover, for any finite $K$, \eqref{GOE_Airy Green fun} holds jointly for $k = 1, 2, \dotsc, K$.
\end{thm}

In the rank-one case, we have the following extension to the G\(\beta\)E, for all \(\beta>0\).   

\begin{thm}[G\texorpdfstring{\(\beta\)}{beta}E with rank-one critical deformation]\label{one_spike_beta}
  Let \(H_N^{(\beta)}\) be the tridiagonal random matrix of dimension \(N\) defined in \eqref{eq:defn_H^beta}. Fix \(\beta>0\) and \(\theta\in\mathbb R\). Let
\begin{equation*}
 % \label{eq:defn_H^beta_N}
\lambda_1\ge \lambda_2\ge \cdots \ge \lambda_N   
\end{equation*}
be the eigenvalues of $H_N^{(\beta)}+d e_1e_1^\top$ defined in \eqref{eq:defn_H^beta+dee},
and let \(\x_k\) be a unit eigenvector associated with \(\lambda_k\). Let \(k\ge1\) be independent of \(N\). Let \(s_{\beta,k}\) be the unique solution in
\[
\bigl(a_{\beta,k},a_{\beta,k-1}\bigr),
\qquad
a_{\beta,0}:=+\infty,
\]
of
\begin{equation} \label{eq:roots_beta_positive}
  \mathcal G^{(\beta)}(x)=\theta,
\end{equation}
where \(\mathcal G^{(\beta)}\) is the Airy-Green function from Proposition~\ref{prop:beta_airy_green}. Then, as \(N\rightarrow\infty\), we have the joint convergence in distribution
\begin{align}\label{eq:limit_rank_one_beta}
 N^{2/3}\bigl(\lambda_k-2\bigr) \Longrightarrow {}& s_{\beta,k}, & N^{1/3}\bigl|\langle e_1,\x_k\rangle\bigr|^2
\Longrightarrow {}&
-\frac{1}{(\mathcal G^{(\beta)})'\!\bigl(s_{\beta,k}\bigr)}.
\end{align}
Moreover, for any finite $K$, \eqref{eq:limit_rank_one_beta} holds jointly for $k = 1, 2, \dotsc, K$.
\end{thm}

The next result shows that as the parameters $\theta_i$'s go to $+\infty$ (resp.~$-\infty$), the eigenvalue and eigenvector overlap properties transit from the critical regime to the subcritical (resp.~supercritical) regime. For simplicity, we consider only the case that one $\theta_i$ goes to $\pm\infty$ if the rank of deformation is $r > 1$.

\begin{pro}[Limits as one \(\theta_i\to\pm\infty\)]
\label{prop:one-coordinate-BBP}
Suppose that \(\beta>0\). Let \(r\) be a fixed positive integer when
\(\beta\in\{1,2\}\), and let \(r=1\) otherwise. Fix
\(i\in\{1,\dots,r\}\). When \(\beta\in\{1,2\}\), let
\(\Theta^{(i)}\) be defined as in \eqref{eq:Theta_i}; otherwise, set
\(\Theta^{(i)}=()\).

The symbols
\(\theta_i\), \(a_{\beta,k}^{\Theta^{(i)}}\),
\(s_{\beta,k}^{(i)}\),
\(\mathcal G^{(\beta)}_{(r-1),i^c}\), and
\(\xi_{\beta,k}^2\) have the meanings specified in
Theorem~\ref{finite_spikes_GOE} for the GOE case with \(\beta=1\)
and the GUE case with \(\beta=2\). For all other \(\beta>0\), they
stand for \(\theta\), \(a_{\beta,k}\), \(s_{\beta,k}\),
\(\mathcal G^{(\beta)}\), and \(\xi_{\beta,k}^2\), respectively, as
in Theorem~\ref{one_spike_beta}. The \(o(1)\) terms in
Parts~\ref{enu:one-coordinate-BBP:positive}--%
\ref{enu:one-coordinate-BBP:negative-k1} are understood almost surely.

\begin{enumerate}[label=\textup{(\roman*)}]

\item\label{enu:one-coordinate-BBP:positive}
For every fixed \(k\ge1\), as \(\theta_i\to+\infty\),
\begin{equation}\label{eq:positive-direct}
\begin{aligned}
s_{\beta,k}^{(i)}
&=
a_{\beta,k}^{\Theta^{(i)}}
+\frac{\beta^{-1}\xi_{\beta,k}^2}{\theta_i}
\bigl(1+o(1)\bigr),
&\qquad
-\frac{1}{
\bigl(\mathcal G^{(\beta)}_{(r-1),i^c}\bigr)'
\bigl(s_{\beta,k}^{(i)}\bigr)}
&=
\frac{\beta^{-1}\xi_{\beta,k}^2}{\theta_i^2}
\bigl(1+o(1)\bigr).
\end{aligned}
\end{equation}

\item % \label{enu:one-coordinate-BBP:negative-kge2}
For every fixed \(k\ge2\), as \(\theta_i\to-\infty\),
\begin{equation}\label{eq:negative-kge2-direct}
\begin{aligned}
s_{\beta,k}^{(i)}
&=
a_{\beta,k-1}^{\Theta^{(i)}}
-\frac{\beta^{-1}\xi_{\beta,k-1}^2}{|\theta_i|}
\bigl(1+o(1)\bigr),
&\qquad
-\frac{1}{
\bigl(\mathcal G^{(\beta)}_{(r-1),i^c}\bigr)'
\bigl(s_{\beta,k}^{(i)}\bigr)}
&=
\frac{\beta^{-1}\xi_{\beta,k-1}^2}{\theta_i^2}
\bigl(1+o(1)\bigr).
\end{aligned}
\end{equation}

\item\label{enu:one-coordinate-BBP:negative-k1}
As \(\theta_i\to-\infty\),
\begin{equation}\label{eq:negative-k1-direct}
\begin{aligned}
s_{\beta,1}^{(i)}
&=
\theta_i^2\bigl(1+o(1)\bigr),
&\qquad
-\frac{1}{
\bigl(\mathcal G^{(\beta)}_{(r-1),i^c}\bigr)'
\bigl(s_{\beta,1}^{(i)}\bigr)}
&=
2|\theta_i|+o(1).
\end{aligned}
\end{equation}

\item % \label{enu:one-coordinate-BBP:normal}
If, in addition, \(\beta\in\{1,2\}\), then, as
\(\theta_i\to-\infty\),
\begin{equation}\label{eq:normal_supercritical}
\begin{aligned}
\frac{s_{\beta,1}^{(i)}-\theta_i^2}
{2(-\theta_i)^{1/2}}
&\Longrightarrow
\mathcal N\!\left(0,\frac1\beta\right),
&\qquad
(-\theta_i)^{1/2}
\left[
-\frac{1}{
\bigl(\mathcal G^{(\beta)}_{(r-1),i^c}\bigr)'
\bigl(s_{\beta,1}^{(i)}\bigr)}
+2\theta_i
\right]
&\Longrightarrow
\mathcal N\!\left(0,\frac2\beta\right).
\end{aligned}
\end{equation}

\end{enumerate}
\end{pro}
\begin{remark}
Consider the finite-rank deformation of GOE/GUE:
\begin{enumerate}
\item
From the two asymptotic relations in
\eqref{eq:positive-direct}, we find that when one critical BBP
transition parameter $d_i=1-\theta_iN^{-1/3}$ moves to the subcritical
side as $\theta_i\to+\infty$, the top eigenvalues $\lambda_k^{(D)}$
behave as if the deformation were of rank $r-1$ specified by
$\Theta^{(i)}$. This is consistent with the eigenvalue-sticking
phenomenon for subcritical finite-rank deformations of Wigner matrices;
see \cite[Theorem~2.7 and Remark~2.13]{Knowles-Yin13}. On the other
hand, the overlap between $\mathbf x_k^{(D)}$ (the eigenvector
associated with $\lambda_k^{(D)}$) and $\mathbf w_i$ (the deformation
vector associated with $d_i$) vanishes, as if $\mathbf w_i$ did not
represent a special direction. We note that if $\mathbf v$ and
$\mathbf w$ are independent random vectors uniformly distributed on
the unit sphere in $\mathbb R^N$ (resp.\ $\mathbb C^N$), then
\[
N\lvert\langle\mathbf v,\mathbf w\rangle\rvert^2
\Longrightarrow
\xi_1^2
\qquad
\left(\text{resp. } \frac12\xi_2^2\right)
\]
as $N\to\infty$. This $\beta^{-1}\xi_\beta^2$ limiting distribution
is also observed in non-outlier eigenvector overlaps in the
subcritical regime of spiked sample covariance matrices; see
\cite[Theorem~2.20]{Bloemendal-Knowles-Yau-Yin16}.

\item
From the two asymptotic relations in
\eqref{eq:negative-k1-direct} and the two convergence statements in
\eqref{eq:normal_supercritical}, we find that as one critical BBP
transition parameter $d_i=1-\theta_iN^{-1/3}$ moves to the
supercritical side as $\theta_i\to-\infty$, the top eigenvalue
$\lambda_1^{(D)}$ moves away from the spectral edge $2$ and becomes an
outlier with Gaussian fluctuations, as featured in the supercritical
regime \cite{Baik-Ben_Arous-Peche05, Peche05, Baik-Silverstein06}. On the other hand, the overlap between $\mathbf x_1^{(D)}$
(the eigenvector associated with $\lambda_1^{(D)}$) and $\mathbf w_i$
(the deformation vector associated with $d_i$) increases with
$|\theta_i|$, and both exhibit Gaussian fluctuations \cite{Capitaine-Donati_Martin-Feral09, Capitaine-Donati_Martin18, Bao-Ding-Wang-Wang20, Bao-Ding-Wang18}. Moreover,
the two asymptotic relations in
\eqref{eq:negative-kge2-direct} show that the other top
eigenvalues behave as if they were the top eigenvalues of a rank-$(r-1)$
deformation specified by $\Theta^{(i)}$, and that the overlaps of their
eigenvectors with $\mathbf w_i$ vanish, as if $\mathbf w_i$ were not a
special direction for them.
\end{enumerate}
\end{remark}

\begin{remark}
Proposition~\ref{prop:one-coordinate-BBP} also applies to the rank-one
deformation of G$\beta$E for general $\beta>0$. However, the Gaussian
fluctuation result for the leading eigenvalue and its associated
eigenvector overlap in the supercritical limit $\theta\to-\infty$ is
established only for $\beta=1,2$, due to the corresponding restriction
in Part~\ref{enu:lem:beta_airy_green_infty:2} of
Lemma~\ref{lem:beta_airy_green_infty}.
\end{remark}
\subsection{Proof strategy and relation to previous results}

The starting point of both the proofs of Theorems \ref{finite_spikes_GOE} and \ref{one_spike_beta} is the following relation between eigenvectors and eigenvalues:

\begin{lem} \label{lem:eigenvector_eigenvalue}
  Let $X$ be an $N$-dimensional real symmetric (resp.~complex Hermitian) matrix with distinct eigenvalues $\mu_1 > \dotsb > \mu_N$, $\w$ be a unit $N$-dimensional column vector in $\realR^N$ (resp.~$\compC^N$), $d$ be a real number, and $Y = X + d \w \w^{\top}$ (resp.~$Y = X + d \w \w^*$) has distinct eigenvalues $\lambda_1 > \dotsb > \lambda_N$. We assume $\{ \mu_j \}^N_{j = 1} \cup \{ \lambda_j \}^N_{j = 1}$ consists of $2N$ distinct real numbers, and for any $k = 1, \dotsc, N$, $\x_k$ is a unit eigenvector of $Y$ associated with $\lambda_k$. Then
  \begin{equation}\label{eq:rank_one_root_identity}
    \w^\top(\lambda_k I - X)^{-1} \w =\frac{1}{d},
  \end{equation}
  and
  \begin{equation} \label{eq:eigenvector_eigenvalue}
    \lvert \langle \x_k, \w \rangle \rvert^2 = \frac{\lvert \langle \w, (\lambda_k I - X)^{-1} \w \rangle \rvert^2}{\langle \w, (\lambda_k I - X)^{-2} \w \rangle} = \frac{1}{d^2} \frac{1}{\langle \w, (\lambda_k I - X)^{-2} \w \rangle}.
  \end{equation}
\end{lem}

In the setting of each of the theorems, identity \eqref{eq:eigenvector_eigenvalue} is rewritten as a differential identity that involves a finite-\(N\) version of the Airy-Green function specified by \eqref{eq:airy_green_beta_limit} or \eqref{eq:defn_Airy_Green_multiple}. Then by showing that the finite-\(N\) version of the Airy-Green function converges to the corresponding Airy-Green function as $N \to \infty$ in the sense of meromorphic convergence that is defined in Definition \ref{def:meromorphic-convergence}, we prove the theorems.

Instead of (\ref{eq:eigenvector_eigenvalue}), another eigenvector-eigenvalue identity \cite{Denton-Parke-Tao-Zhang19} is used as  the starting point of \cite[Theorem 2]{Bao-Wang20} for the eigenvector overlap in the critical regime of BBP transition for the GUE case.  
It is worth noting that since our paper is based on Lemma \ref{lem:eigenvector_eigenvalue} that is different from the eigenvector-eigenvalue identity from \cite{Denton-Parke-Tao-Zhang19}  on which \cite{Bao-Wang20} is based, the GUE case of Theorem \ref{finite_spikes_GOE} and \cite[Theorem 2]{Bao-Wang20} express the same limit in quite different forms. Although the two formulas describe the same limit law, their direct equivalence is not transparent.

\subsection{Possible Extensions}

The method is expected to extend to critically spiked real and complex Wishart ensembles which are also known as Laguerre orthogonal and unitary ensembles, and also the Laguerre $\beta$-ensemble with rank one spike. Extension to other spiked models, like the factor model and canonical correlation analysis that are considered in \cite{Bykhovskaya-Gorin-Sodin25} is also possible, and we leave the discussion to future work.

\subsubsection*{Acknowledgment}
Zhigang Bao was partially supported by Hong Kong RGC Grant GRF 16304724 and 17304225. 
Dong Wang and Yue Zhu were partially supported by the National Natural Science Foundation of China under grant number 12271502, and the University of Chinese Academy of Sciences start-up grant 118900M043.

\section{Proofs of the main results}

We begin by recording several standard spectral facts used repeatedly below.

\begin{fact} \label{rem:standard_spectral_facts}
  We use the following fact throughout.

  \begin{enumerate}
  \item\label{enu:rem:standard_spectral_facts_GOE_GUE}
    Let \(H_N\) be a GOE or GUE random matrix, and let \(\mathbb F=\mathbb R\) in the GOE case and \(\mathbb F=\mathbb C\) in the GUE case. Let \(P\) be a deterministic real symmetric or Hermitian matrix, and let \(\w\in\mathbb F^N\) be a deterministic unit vector and $W = \w\w^{\top}$ in the GOE case and $W = \w\w^{\ast}$ in the GUE case. Then, for every \(d\in\mathbb R\setminus\{0\}\),
    \[
      \sigma(H_N+P)\cap\sigma(H_N+P+d W)=\varnothing
      \qquad\text{almost surely}.
    \]
    Moreover, \(H_N+P\) has simple spectrum almost surely. These facts follow from the absolute continuity of the Gaussian entries and the standard discriminant/resultant polynomial argument; see \cite[Corollary~2.5.4 and Lemma~2.5.5]{Anderson-Guionnet-Zeitouni10} and \cite{Tao-Vu17}.
    
  \item\label{enu:rem:standard_spectral_facts_3}
    For the Dumitriu--Edelman tridiagonal G\(\beta\)E, the subdiagonal entries are almost surely positive. Hence \(H_N^{(\beta)}\) is an irreducible Jacobi matrix, so its spectrum is almost surely simple and the first coordinate of each eigenvector is nonzero. Consequently, for every \(d\in\mathbb R\setminus\{0\}\),
    \[
      \sigma(H_N^{(\beta)})\cap
      \sigma(H_N^{(\beta)}+d\,e_1e_1^\top)=\varnothing
      \qquad\text{almost surely}.
    \]
    Moreover, if \(\u_1,\dots,\u_N\) are the normalized eigenvectors of \(H_N^{(\beta)}\), then the vector of squared first coordinates is independent of the eigenvalues and satisfies
    \[
      \bigl(
      |\langle e_1,\u_1\rangle|^2,\dots,
      |\langle e_1,\u_N\rangle|^2
      \bigr)
      \overset{d}{=}
      \left(
        \frac{\xi^2_1}{\sum_{\ell=1}^N\xi^2_\ell},\dots,
        \frac{\xi^2_N}{\sum_{\ell=1}^N\xi^2_\ell}
      \right),
    \]
    where \(\xi_1,\dots,\xi_N\) are i.i.d.\ random variables with \(\xi^2_j\sim\chi_\beta^2\), independent of the eigenvalues; see \cite[Theorem~2.12]{Dumitriu-Edelman02} and \cite[Theorem~5.2.1]{Dumitriu03}.
  \end{enumerate}
\end{fact}

In this paper, we adopt the following notion of meromorphic convergence that will be used in the sequel:

\begin{defn}\label{def:meromorphic-convergence}
  Let \(\Omega\subset\mathbb C\) be an open set, and let \( \{ f_N \}^{\infty}_{N = 1} \) and \(f\) be meromorphic functions on \(\Omega\). We write
  \[
    f_N\xrightarrow{\mathrm{mer}} f
    \qquad\text{on }\Omega
  \]
  if, for every compact set \(K\subset\Omega\) containing no pole of \(f\), there exists \(N_0=N_0(K)\) such that \(K\) contains no pole of \(f_N\) for all \(N\ge N_0\), and
\[
\sup_{z\in K}|f_N(z)-f(z)|\to0
\qquad\text{as }N\to\infty .
\]
\end{defn}

This notion agrees with locally uniform convergence in the spherical
metric used in \cite[Section~8]{Bykhovskaya-Gorin-Sodin25}. The
meromorphic functions considered in this paper belong to the class
$\Omega_-$ of \cite[Definition~8.10]{Bykhovskaya-Gorin-Sodin25}; in
particular, they have simple real poles, and their negatives are
Herglotz--Nevanlinna functions.

By Cauchy's integral formula, \(f_N\xrightarrow{\mathrm{mer}}f\) implies \(f_N'\xrightarrow{\mathrm{mer}}f'\) on the same domain. We shall use the following elementary consequence.

\begin{lem} \label{lem:simple-zero-stability-meromorphic}
  Let \(U\subset\mathbb C\) be an open set, let \(x_0\in U\cap\mathbb R\), and assume that \(\{f_N\}_{N=1}^\infty\xrightarrow{\mathrm{mer}} f\) on \(U\). Suppose that \(x_0\) is not a pole of \(f\), and
  \begin{align*}
    f(x_0)= {}& 0, & f'(x_0) \neq {}& 0 .
  \end{align*}
  Then there exists \(\delta>0\) such that, for all sufficiently large \(N\), \(f_N\) has exactly one zero \(z_N\) in the disc \(\{z:|z-x_0|<\delta\}\), and \(z_N\to x_0\). Moreover,
  \[
    f_N'(z_N)\to f'(x_0).
  \]
  If, in addition, \(f_N(\overline z)=\overline{f_N(z)}\) on this disc, then \(z_N\in\mathbb R\).
\end{lem}

\begin{proof}
  Choose \(\delta>0\) such that the closed disc \(\overline D(x_0,\delta)\subset U\), contains no pole of \(f\), and contains no zero of \(f\) except \(x_0\). Since \(x_0\) is a simple zero, \(f\) has no zero on \(\partial D(x_0,\delta)\). By uniform convergence on \(\partial D(x_0,\delta)\), Rouch\'e's theorem implies that \(f_N\) and \(f\) have the same number of zeros in \(D(x_0,\delta)\), counted with multiplicity, for all large \(N\). Hence, \(f_N\) has exactly one zero \(z_N\) there. Shrinking \(\delta\) proves \(z_N\to x_0\). The derivative convergence follows from Cauchy's integral formula. Finally, if \(f_N(\overline z)=\overline{f_N(z)}\), then \(\overline{z_N}\) is also a zero in the same disc; uniqueness gives \(z_N=\overline{z_N}\).
\end{proof}

\subsection{Proof of the GOE case of Theorem~\ref{finite_spikes_GOE} with rank \texorpdfstring{$1$}{1}}\label{subsec:rank_1_GOE}

We begin with the GOE case of Theorem~\ref{finite_spikes_GOE} with $r = 1$ and $i = 1$. To simplify the notation, we write \((\theta_1,d_1)\) as \((\theta,d)\), and rewrite \(s_{1, k}^{(1)}\), \(H_N^{(D)}\), \(\lambda_k^{(D)}\), and \(\x_k^{(D)}\) as \(s_k\), \(H_N^{(d)}\), \(\lambda_k\), and \(\x_k\), respectively. We also write, with \(\Theta^{(i)}=()\)
\begin{align*}
  a_j:= {}& a_{1,j}=a_{1,j}^{\Theta^{(i)}}, & \mathcal G:= {}& \mathcal G^{(1)}=\mathcal G^{(1)}_{\Theta^{(i)}}.
\end{align*}
Also due to the orthogonal invariance of $H_N$, we assume $\w_1 = e_1$ without loss of generality.

Let \( \mu_1\ge \mu_2\ge \cdots \ge \mu_N \) as in \eqref{matrix eigva} be the eigenvalues of \(H_N\). By Part \ref{enu:rem:standard_spectral_facts_GOE_GUE} of Fact~\ref{rem:standard_spectral_facts}, we assume that \( \{\mu_1,\dots,\mu_N\}\cup\{\lambda_1,\dots,\lambda_N\} \) consists of \(2N\) distinct real numbers, which holds almost surely. From the orthogonal invariance of $H_N$, we have
\begin{equation*} % \label{eq:goe_spectral_decomp}
  H_N=U\diag(\mu_1,\dots,\mu_N)U^\top, \qquad U=(\u_1,\dots,\u_N),
\end{equation*}
where \(U\) is Haar distributed and independent of \((\mu_1,\dots,\mu_N)\), and \(\mathbf{u}_k\) is a unit eigenvector associated with \(\mu_k\); see \cite[Corollary~2.5.4]{Anderson-Guionnet-Zeitouni10}. Because \(U\) is Haar distributed and independent of the eigenvalues, its first row is a uniformly distributed unit vector, and there are i.i.d.\ standard Gaussian random variables \(\xi_{N,1},\dots,\xi_{N,N}\),
independent of \((\mu_1,\dots,\mu_N)\), such that
\begin{equation*}
    % \label{(4.5)}
    (U^\top e_1 )_i=\langle \u_i,e_1 \rangle
 = 
\frac{\xi_{N,i}}{\bigl(\sum_{j=1}^N \xi_{N,j}^2\bigr)^{1/2}},
\qquad i=1,\dots,N.
\end{equation*}
Accordingly, define the random meromorphic function
\begin{equation*} % \label{eq:defn_Gtilde_N}
  \begin{split}
    \widetilde{\mathcal G}_N(w) :={}& N^{1/3}\,e_1 ^\top\!\bigl((2+N^{-2/3}w)I-H_N\bigr)^{-1}e_1 -N^{1/3} \\
    ={}& \frac{N}{\sum_{i=1}^N \xi_{N,i}^2} \sum_{i=1}^N \frac{\xi_{N,i}^2}{w-N^{2/3}(\mu_i-2)} -N^{1/3}.
\end{split}
\end{equation*}
Then \eqref{eq:rank_one_root_identity} in Lemma \ref{lem:eigenvector_eigenvalue} with $\w = e_1$ implies
\begin{equation*} % \label{eq:s_kN}
  \widetilde{\mathcal G}_N(s_{k,N}) = N^{1/3}\Bigl(\frac1d-1\Bigr) = \frac{\theta}{d}, \quad \text{where} \quad s_{k,N}:=N^{2/3}(\lambda_k-2),
\end{equation*}
and the zeros of \( \widetilde{\mathcal G}_N(w) - \theta/d \) are precisely \(\{s_{k,N}\}_{k=1}^N\).

Since \(d>0\) for all sufficiently large \(N\), the rank-one interlacing
inequality gives
\[
\lambda_k\in(\mu_k,\mu_{k-1}),
\qquad k\ge1,
\]
with the convention \(\mu_0=+\infty\). Equivalently,
\[
s_{k,N}\in
\bigl(
N^{2/3}(\mu_k-2),
N^{2/3}(\mu_{k-1}-2)
\bigr).
\]
Thus \(s_{k,N}\) is the unique zero of \(\widetilde{\mathcal G}_N(w)-\theta/d\) in the \(k\)-th interval between two consecutive poles.

Differentiating \(\widetilde{\mathcal G}_N\), we obtain
\[
  (\widetilde{\mathcal G}_N)'(w) = -N^{-1/3}e_1^\top\bigl((2+N^{-2/3}w)I-H_N\bigr)^{-2}e_1.
\]
Evaluating \((\widetilde{\mathcal G}_N)'(w)\) at \(w=s_{k,N}\), we have
\begin{equation*} % \label{eq:G'_at_s_kN}
  (\widetilde{\mathcal G}_N)'(s_{k,N}) = -N^{-1/3} e_1 ^\top(\lambda_k I-H_N)^{-2}e_1 .
\end{equation*}
By \eqref{eq:eigenvector_eigenvalue} in Lemma \ref{lem:eigenvector_eigenvalue} with $\w = e_1$, we find
\begin{equation} \label{eq:overlap_formula_true}
  |\langle e_1 ,\x_k\rangle|^2 = -\frac{1}{N^{1/3}}\frac{1}{d^2\,(\widetilde{\mathcal G}_N)'(s_{k,N})}.
\end{equation}

Next, write
\begin{equation} \label{eq:tildeG_vs_G_rank1}
\widetilde{\mathcal G}_N(w)
=
c_N\mathcal G_N(w)+(c_N-1)N^{1/3},
\qquad
c_N:=\frac{N}{\sum_{i=1}^N \xi_{N,i}^2},
\end{equation}
where
\[
\mathcal G_N(w)
:=
\sum_{i=1}^N
\frac{\xi_{N,i}^2}{w-N^{2/3}(\mu_i-2)}
-N^{1/3}
\]
is the finite-\(N\) version of the Airy-Green function in the sense of \cite[Definition~8.10 and Theorem~8.20]{Bykhovskaya-Gorin-Sodin25}, equivalently the function in \eqref{eq:G_beta_N} with \(\beta=1\). Since \(\xi_{N,i}^2\) are independent \(\chi_1^2\) variables, we have \((c_N-1)N^{1/3}\to0\) almost surely. On the other hand, by \cite[Theorem~8.20]{Bykhovskaya-Gorin-Sodin25}, or equivalently by Proposition~\ref{enu:prop:beta_airy_green_2} with \(\beta=1\), there exists a coupling under which the functions \(\{\mathcal{G}_N\}_{N=1}^\infty\) are defined on the same probability space and
\[
\mathcal G_N\xrightarrow{\mathrm{mer}}\mathcal G
\qquad\text{almost surely},
\]
where \(\mathcal G=\mathcal G^{(1)}\) is the Airy-Green function in Proposition~\ref{prop:beta_airy_green}. Hence, \eqref{eq:tildeG_vs_G_rank1} implies that with respect to the same coupling,
\begin{equation} \label{eq:convergence_G_N_G}
\widetilde{\mathcal G}_N\xrightarrow[]{\mathrm{mer}}\mathcal G
\qquad\text{almost surely.}
\end{equation}

Recall that $\mathcal G$ is defined by the Airy$_1$ point process
$\{a_j\}_{j=1}^{\infty}$ and the random variables
$\{\xi_j^2\}_{j=1}^{\infty}$, which are i.i.d.\ with
$\chi_1^2$ distribution and are independent of
$\{a_j\}_{j=1}^{\infty}$. By
\cite[Lemma~2.7]{Bloemendal-Virag11}, the Dirichlet spectrum of the
scalar stochastic Airy operator is almost surely simple. Hence,
\[
a_1>a_2>\cdots
\qquad\text{almost surely}.
\]

Conditioned on $\{a_j\}_{j=1}^{\infty}$, for every
$x\in\mathbb R\setminus\{a_j:j\ge1\}$,
\begin{equation*} % \label{rank one limit function der}
\mathcal G'(x)
=
-\sum_{j=1}^{\infty}
\frac{\xi_j^2}{(x-a_j)^2}
<0
\qquad\text{almost surely}.
\end{equation*}
Thus, $\mathcal G$ is strictly decreasing on each interval between
two consecutive Airy points. Moreover, for every $j\ge2$,
\[
\lim_{x\downarrow a_j}\mathcal G(x)=+\infty,
\qquad
\lim_{x\uparrow a_{j-1}}\mathcal G(x)=-\infty.
\]
By \cite[Proposition~8.3]{Bykhovskaya-Gorin-Sodin25}, or equivalently
by the $\beta=1$ case of
Part~\ref{enu:lem:beta_airy_green_infty:1} of
Lemma~\ref{lem:beta_airy_green_infty}, almost surely,
\[
\mathcal G(x)\longrightarrow-\infty
\qquad\text{as }x\to+\infty.
\]
Therefore, for every $\theta\in\mathbb R$, the equation
\[
\mathcal G(x)=\theta
\]
has a unique real solution in each interval
\[
(a_k,a_{k-1}),\qquad k\ge1,
\]
where $a_0:=+\infty$. In particular, the solutions
$s_k=s_k^{(1)}$ appearing in
Theorem~\ref{finite_spikes_GOE} are well defined and satisfy
\[
s_1>s_2>\cdots
\qquad\text{almost surely}.
\]

Since $\theta/d\to\theta$ as $N\to\infty$, it follows from
\eqref{eq:convergence_G_N_G} and
Lemma~\ref{lem:simple-zero-stability-meromorphic} that, under the
coupling constructed above,
\begin{align}
s_{k,N}
&\longrightarrow s_k,
\label{eq:rank_one_GOE_root_convergence}\\
(\widetilde{\mathcal G}_N)'(s_{k,N})
&\longrightarrow \mathcal G'(s_k)
\label{eq:rank_one_GOE_derivative_convergence}
\end{align}
almost surely. Moreover, since the meromorphic convergence in
\eqref{eq:convergence_G_N_G} holds on a single event of probability
one, \eqref{eq:rank_one_GOE_root_convergence} and
\eqref{eq:rank_one_GOE_derivative_convergence} hold simultaneously
for all \(k\ge1\). Since \(d\to1\), the exact overlap identity
\eqref{eq:overlap_formula_true} then implies that the convergence in
\eqref{GOE_Airy Green fun} holds jointly for every finite \(K\).
This completes the proof of the GOE case of
Theorem~\ref{finite_spikes_GOE} with rank \(1\).

\subsection{Proof of Theorem~\ref{one_spike_beta}}

The argument is nearly the same as the proof of the rank-one GOE case
of Theorem~\ref{finite_spikes_GOE}. Although the eigenvectors of the
tridiagonal model are no longer uniformly distributed, the
distribution of their first coordinates is given in
\cite[Theorem~5.2.1]{Dumitriu03}, which suffices for our analysis.

Recall that
\[
\lambda_1\ge\cdots\ge\lambda_N
\]
are the eigenvalues of \(H_N^{(\beta)}+d e_1e_1^\top\), and let
\(\x_k\) be a normalized eigenvector associated with \(\lambda_k\).
Also, let
\[
\mu_1^{(\beta)}>\cdots>\mu_N^{(\beta)}
\]
be the eigenvalues of \(H_N^{(\beta)}\), as defined in
\eqref{eq:defn_mu^beta_k}. By
Part~\ref{enu:rem:standard_spectral_facts_3} of
Fact~\ref{rem:standard_spectral_facts}, we may work on the event of
probability one on which
\[
\{\mu_1^{(\beta)},\dots,\mu_N^{(\beta)}\}
\cup
\{\lambda_1,\dots,\lambda_N\}
\]
consists of \(2N\) distinct real numbers.

Using Lemma~\ref{lem:eigenvector_eigenvalue}, we obtain
\begin{equation}\label{eq:beta_exact_overlap}
|\langle e_1,\x_k\rangle|^2
=
\frac{1}{d^2}\,
\frac{1}{
e_1^\top(\lambda_k I-H_N^{(\beta)})^{-2}e_1
}.
\end{equation}

Define
\[
\widetilde{\mathcal G}_N^{(\beta)}(w)
:=
N^{1/3}
\left(
e_1^\top
\bigl((2+N^{-2/3}w)I-H_N^{(\beta)}\bigr)^{-1}
e_1-1
\right).
\]
By \eqref{eq:rank_one_root_identity} in
Lemma~\ref{lem:eigenvector_eigenvalue}, with \(\w=e_1\), we have
\begin{equation*} % \label{eq:beta_true_root_eqn}
\widetilde{\mathcal G}_N^{(\beta)}(s_{\beta,k,N})
=
N^{1/3}\left(\frac1d-1\right)
=
\frac{\theta}{d},
\qquad
s_{\beta,k,N}:=N^{2/3}(\lambda_k-2).
\end{equation*}
The zeros of
\[
\widetilde{\mathcal G}_N^{(\beta)}(w)-\frac{\theta}{d}
\]
are precisely \(\{s_{\beta,k,N}\}_{k=1}^N\).

Since \(d>0\) for all sufficiently large \(N\), rank-one
interlacing gives
\[
\lambda_k\in
\bigl(\mu_k^{(\beta)},\mu_{k-1}^{(\beta)}\bigr),
\qquad k\ge1,
\]
where \(\mu_0^{(\beta)}:=+\infty\). Equivalently,
\[
s_{\beta,k,N}
\in
\left(
N^{2/3}(\mu_k^{(\beta)}-2),
N^{2/3}(\mu_{k-1}^{(\beta)}-2)
\right).
\]
Thus, \(s_{\beta,k,N}\) is the unique zero of
\(\widetilde{\mathcal G}_N^{(\beta)}(w)-\theta/d\) in the \(k\)-th
interval between two consecutive poles.

Differentiating \(\widetilde{\mathcal G}_N^{(\beta)}\) and evaluating
at \(s_{\beta,k,N}\), we obtain
\[
\bigl(\widetilde{\mathcal G}_N^{(\beta)}\bigr)'
(s_{\beta,k,N})
=
-N^{-1/3}
e_1^\top
(\lambda_k I-H_N^{(\beta)})^{-2}
e_1.
\]
Together with \eqref{eq:beta_exact_overlap}, this gives
\begin{equation}\label{eq:beta_overlap_true}
|\langle e_1,\x_k\rangle|^2
=
-\frac{N^{-1/3}}{
d^2
\bigl(\widetilde{\mathcal G}_N^{(\beta)}\bigr)'
(s_{\beta,k,N})
}.
\end{equation}

Write the spectral decomposition as
\[
H_N^{(\beta)}
=
\sum_{i=1}^N
\mu_i^{(\beta)}
\u_{\beta,i}\u_{\beta,i}^\top.
\]
By \cite[Theorem~5.2.1]{Dumitriu03}, there exist i.i.d.\ random
variables
\[
\xi_{\beta,N,1}^2,\dots,\xi_{\beta,N,N}^2
\sim\chi_\beta^2,
\]
independent of \(\{\mu_i^{(\beta)}\}_{i=1}^N\), such that
\begin{equation*} % \label{eq:beta_eigenvector_weights}
\begin{split}
&\bigl(
|\langle e_1,\u_{\beta,1}\rangle|^2,
\dots,
|\langle e_1,\u_{\beta,N}\rangle|^2
\bigr)
\\
&\qquad\overset{d}{=}
\left(
\frac{\beta^{-1}\xi_{\beta,N,1}^2}{
\sum_{j=1}^N\beta^{-1}\xi_{\beta,N,j}^2
},
\dots,
\frac{\beta^{-1}\xi_{\beta,N,N}^2}{
\sum_{j=1}^N\beta^{-1}\xi_{\beta,N,j}^2
}
\right).
\end{split}
\end{equation*}

Define
\begin{equation}\label{eq:G_beta_N}
\mathcal G_N^{(\beta)}(w)
:=
\sum_{i=1}^N
\frac{\beta^{-1}\xi_{\beta,N,i}^2}{
w-N^{2/3}(\mu_i^{(\beta)}-2)
}
-N^{1/3}.
\end{equation}
Then, analogously to \eqref{eq:tildeG_vs_G_rank1},
\begin{equation*} % \label{eq:tildeG_vs_G_beta}
\widetilde{\mathcal G}_N^{(\beta)}(w)
=
c_{\beta,N}\mathcal G_N^{(\beta)}(w)
+
(c_{\beta,N}-1)N^{1/3},
\qquad
c_{\beta,N}
:=
\frac{N}{
\sum_{i=1}^N\beta^{-1}\xi_{\beta,N,i}^2
}.
\end{equation*}
Moreover,
\[
(c_{\beta,N}-1)N^{1/3}
\longrightarrow0
\qquad\text{almost surely}.
\]

We use the following extension of
\cite[Theorem~8.20]{Bykhovskaya-Gorin-Sodin25} from \(\beta=1\) to
all \(\beta>0\).

\begin{pro}\label{enu:prop:beta_airy_green_2}
Suppose that
\(\{\xi_{\beta,N,i}\}_{i=1}^N\) are independent of
\(\{\mu_i^{(\beta)}\}_{i=1}^N\) and satisfy
\(\xi_{\beta,N,i}^2\sim\chi_\beta^2\). There exists a coupling that
places
\[
\left(
\{\mu_i^{(\beta)}\}_{i=1}^N,
\{\xi_{\beta,N,i}^2\}_{i=1}^N
\right)_{N=1}^{\infty}
\]
and \(\mathcal G^{(\beta)}\) on the same probability space such that,
almost surely,
\begin{equation}
\label{eq:beta_airy_green_function_convergence}
\mathcal G_N^{(\beta)}
\xrightarrow{\mathrm{mer}}
\mathcal G^{(\beta)}.
\end{equation}
\end{pro}

The proof of Proposition~\ref{enu:prop:beta_airy_green_2} is given in
Section~\ref{subsec:proof_enu:prop:beta_airy_green_2}.

Recall that \(\mathcal G^{(\beta)}\) is defined in
\cite[Remark~4.5]{Bykhovskaya-Gorin-Sodin25} and
\eqref{eq:airy_green_beta_limit} by the Airy\(_\beta\) point process
\(\{a_{\beta,j}\}_{j=1}^{\infty}\) and the random variables
\(\{\xi_{\beta,j}^2\}_{j=1}^{\infty}\), which are i.i.d.\ with
\(\chi_\beta^2\) distribution and are independent of
\(\{a_{\beta,j}\}_{j=1}^{\infty}\).

By \cite[Lemma~2.7]{Bloemendal-Virag11}, the Dirichlet spectrum of the
scalar stochastic Airy operator is almost surely simple. Hence,
\[
a_{\beta,1}>a_{\beta,2}>\cdots
\qquad\text{almost surely}.
\]

Conditioned on \(\{a_{\beta,j}\}_{j=1}^{\infty}\), for every
\(x\in\mathbb R\setminus\{a_{\beta,j}:j\ge1\}\),
\begin{equation*} % \label{rank one beta function limit}
\bigl(\mathcal G^{(\beta)}\bigr)'(x)
=
-\sum_{j=1}^{\infty}
\frac{\beta^{-1}\xi_{\beta,j}^2}{
(x-a_{\beta,j})^2
}
<0
\qquad\text{almost surely}.
\end{equation*}
Thus, \(\mathcal G^{(\beta)}\) is strictly decreasing on each interval
between two consecutive Airy\(_\beta\) points. Moreover, for every
\(j\ge2\),
\[
\lim_{x\downarrow a_{\beta,j}}
\mathcal G^{(\beta)}(x)
=
+\infty,
\qquad
\lim_{x\uparrow a_{\beta,j-1}}
\mathcal G^{(\beta)}(x)
=
-\infty.
\]
By Part~\ref{enu:lem:beta_airy_green_infty:1} of
Lemma~\ref{lem:beta_airy_green_infty}, almost surely,
\[
\mathcal G^{(\beta)}(x)
\longrightarrow-\infty
\qquad\text{as }x\to+\infty.
\]
Therefore, for every \(\theta\in\mathbb R\), the equation
\[
\mathcal G^{(\beta)}(x)=\theta
\]
has a unique real solution
\[
s_{\beta,k}
\in
(a_{\beta,k},a_{\beta,k-1}),
\qquad k\ge1,
\]
where \(a_{\beta,0}:=+\infty\).

Since \(\theta/d\to\theta\) as \(N\to\infty\), it follows from
\eqref{eq:beta_airy_green_function_convergence} and
Lemma~\ref{lem:simple-zero-stability-meromorphic} that
\begin{align}
s_{\beta,k,N}
&\longrightarrow s_{\beta,k},
\label{eq:beta_root_convergence}\\
\bigl(\widetilde{\mathcal G}_N^{(\beta)}\bigr)'
(s_{\beta,k,N})
&\longrightarrow
\bigl(\mathcal G^{(\beta)}\bigr)'(s_{\beta,k})
\label{eq:beta_derivative_convergence}
\end{align}
almost surely under the coupling above. Moreover, since the
meromorphic convergence holds on a single event of probability one,
\eqref{eq:beta_root_convergence} and
\eqref{eq:beta_derivative_convergence} hold simultaneously for all
\(k\ge1\). Since \(d\to1\), the exact overlap identity
\eqref{eq:beta_overlap_true} then implies that the convergence in
\eqref{eq:limit_rank_one_beta} holds jointly for every finite \(K\).
This completes the proof of Theorem~\ref{one_spike_beta}.
\subsection{Proof of the GOE case of
Theorem~\ref{finite_spikes_GOE} by induction on the rank}
\label{subsec:proof_mult_rank_GOE}

We prove by induction on \(r\) the overlap convergence in
Theorem~\ref{finite_spikes_GOE}, simultaneously with the almost-sure
simplicity of the higher-rank Airy\(_1\) point process
\[
\bigl\{
a_{1,k}^{\Theta}
\bigr\}_{k\ge1}, \quad \Theta = (\theta_1,\dots,\theta_r),
\]
that is,
\begin{equation*} % \label{eq:simplicity_higher_rank_Airy_GOE}
a_{1,1}^{\Theta}
>
a_{1,2}^{\Theta}
>
\cdots
\qquad\text{almost surely}.
\end{equation*}

For \(r=1\), the point process
\(\{a_{1,k}^{(\theta_1)}\}_{k\ge1}\) is the negative spectrum of the
scalar stochastic Airy operator with Robin boundary condition
\(f'(0)=\theta_1f(0)\). By
\cite[Lemma~2.7]{Bloemendal-Virag11}, every eigenspace of this
operator is at most one-dimensional. Therefore,
\[
a_{1,1}^{(\theta_1)}
>
a_{1,2}^{(\theta_1)}
>
\cdots
\qquad\text{almost surely}.
\]
The rank-one overlap convergence was proved in
Section~\ref{subsec:rank_1_GOE}, so the induction base is complete.

Assume now that both assertions hold for all ranks up to \(r-1\), and
consider rank \(r\). By orthogonal invariance and relabeling the
spikes, we may assume
\[
\w_j=e_j,\qquad 1\le j\le r,
\qquad\text{and}\qquad i=r.
\]
For notational simplicity, set
\begin{equation}\label{eq:notation_mult_GOE}
\begin{aligned}
\Theta^{(r)}
&:=(\theta_1,\dots,\theta_{r-1}),
&
a_k^{\Theta^{(r)}}
&:=a_{1,k}^{\Theta^{(r)}},
\\
s_k^{(r)}
&:=s_{1,k}^{(r)},
&
\mathcal G_{r-1}
&:=\mathcal G^{(1)}_{(r-1),r^c}.
\end{aligned}
\end{equation}
By the induction hypothesis, the rank-\((r-1)\) Airy point process is
almost surely simple:
\begin{equation*} % \label{eq:induction_simplicity_rank_r_minus_one}
a_1^{\Theta^{(r)}}
>
a_2^{\Theta^{(r)}}
>
\cdots
\qquad\text{almost surely}.
\end{equation*}

Define
\[
\widehat H_N
:=
H_N+\sum_{j=1}^{r-1}d_je_je_j^\top,
\qquad
H_N^{(D)}
=
\widehat H_N+d_re_re_r^\top.
\]
Let
\[
\widehat\mu_{1,N}\ge\cdots\ge\widehat\mu_{N,N}
\]
be the eigenvalues of \(\widehat H_N\), with corresponding normalized
eigenvectors
\(\widehat\u_{1,N},\dots,\widehat\u_{N,N}\). By
Part~\ref{enu:rem:standard_spectral_facts_GOE_GUE} of
Fact~\ref{rem:standard_spectral_facts}, the spectra of
\(\widehat H_N\) and \(H_N^{(D)}\) are almost surely simple and
disjoint. The choices of signs of the eigenvectors are immaterial,
since only their squared coordinates are used below.

Put
\[
a_{m,N}:=N^{2/3}(\widehat\mu_{m,N}-2)
\]
and define
\begin{equation}\label{eq:defn_Gtilde_N_r}
\begin{aligned}
\widetilde{\mathcal G}_{N,r}(w)
&:=
N^{1/3}
\left[
e_r^\top
\bigl((2+N^{-2/3}w)I-\widehat H_N\bigr)^{-1}
e_r-1
\right]
\\
&=
\sum_{m=1}^N
\frac{
N|\langle e_r,\widehat\u_{m,N}\rangle|^2
}{
w-a_{m,N}
}
-N^{1/3}.
\end{aligned}
\end{equation}

By \eqref{eq:rank_one_root_identity} in
Lemma~\ref{lem:eigenvector_eigenvalue},
\begin{equation*} % \label{eq:rank_r_root_equation_finite_N}
\widetilde{\mathcal G}_{N,r}(s_{k,N}^{(r)})
=
\frac{\theta_r}{d_r},
\qquad
s_{k,N}^{(r)}
:=
N^{2/3}(\lambda_k^{(D)}-2).
\end{equation*}
Since \(d_r>0\) for all sufficiently large \(N\), rank-one
interlacing gives
\[
s_{k,N}^{(r)}
\in
(a_{k,N},a_{k-1,N}),
\qquad
a_{0,N}:=+\infty.
\]
Thus, \(s_{k,N}^{(r)}\) is the unique zero of
\(\widetilde{\mathcal G}_{N,r}-\theta_r/d_r\) in this interval.

The exact overlap identity in
Lemma~\ref{lem:eigenvector_eigenvalue} gives
\begin{equation}\label{eq:general-overlap-exact}
|\langle e_r,\x_k^{(D)}\rangle|^2
=
-\frac{N^{-1/3}}{
d_r^2\,
\widetilde{\mathcal G}_{N,r}'(s_{k,N}^{(r)})
}.
\end{equation}

We use the following joint edge convergence.

\begin{lem}\label{lem:joint-edge-gaussian-general-short}
For every fixed \(M\ge1\),
\begin{equation}\label{eq:joint-edge-weight-convergence}
\left(
(a_{m,N})_{m=1}^M,\,
\bigl(
N|\langle e_r,\widehat\u_{m,N}\rangle|^2
\bigr)_{m=1}^M
\right)
\Longrightarrow
\left(
(a_m^{\Theta^{(r)}})_{m=1}^M,\,
(\eta_m^2)_{m=1}^M
\right),
\end{equation}
where \(\eta_1,\dots,\eta_M\) are i.i.d.\ standard Gaussian random
variables, independent of
\(\{a_m^{\Theta^{(r)}}\}_{m\ge1}\).
\end{lem}

As a consequence, we obtain the higher-rank analogue of
\cite[Theorem~8.20]{Bykhovskaya-Gorin-Sodin25}.

\begin{pro}\label{prop:general-meromorphic}
There exists a coupling under which
\(\widetilde{\mathcal G}_{N,r}\) and \(\mathcal G_{r-1}\) are defined
on the same probability space and
\begin{equation*} % \label{eq:convergence_G_N_G_r}
\widetilde{\mathcal G}_{N,r}
\xrightarrow{\mathrm{mer}}
\mathcal G_{r-1}
\qquad\text{almost surely}.
\end{equation*}
\end{pro}

The proof of Proposition~\ref{prop:general-meromorphic} is given in
Section~\ref{subsec:proof_prop:general-meromorphic}. We first complete
the induction argument and then prove
Lemma~\ref{lem:joint-edge-gaussian-general-short}.

Recall that \(\mathcal G_{r-1}\) is the Airy--Green function defined
from the rank-\((r-1)\) Airy point process
\[
\{a_k^{\Theta^{(r)}}\}_{k\ge1}
\]
and an independent sequence
\(\{\xi_k^2\}_{k\ge1}\) of \(\chi_1^2\) random variables. Conditioned
on the Airy points,
\[
\mathcal G_{r-1}'(x)
=
-\sum_{j=1}^{\infty}
\frac{\xi_j^2}{
(x-a_j^{\Theta^{(r)}})^2
}
<0
\]
for every
\(x\notin\{a_j^{\Theta^{(r)}}:j\ge1\}\), almost surely. Moreover, for
every \(j\ge2\),
\[
\lim_{x\downarrow a_j^{\Theta^{(r)}}}
\mathcal G_{r-1}(x)
=
+\infty,
\qquad
\lim_{x\uparrow a_{j-1}^{\Theta^{(r)}}}
\mathcal G_{r-1}(x)
=
-\infty,
\]
and Part~\ref{enu:lem:beta_airy_green_infty:1} of
Lemma~\ref{lem:beta_airy_green_infty} gives
\[
\mathcal G_{r-1}(x)
\longrightarrow-\infty
\qquad\text{as }x\to+\infty.
\]

Consequently, for every \(k\ge1\), the equation
\[
\mathcal G_{r-1}(x)=\theta_r
\]
has a unique solution
\begin{equation*} % \label{eq:rank_r_Airy_Green_roots}
s_k^{(r)}
\in
\bigl(
a_k^{\Theta^{(r)}},
a_{k-1}^{\Theta^{(r)}}
\bigr),
\qquad
a_0^{\Theta^{(r)}}:=+\infty.
\end{equation*}
In particular,
\[
s_1^{(r)}>s_2^{(r)}>\cdots
\qquad\text{almost surely}.
\]

Since \(\theta_r/d_r\to\theta_r\), Proposition
\ref{prop:general-meromorphic} and
Lemma~\ref{lem:simple-zero-stability-meromorphic} imply, under the
coupling above,
\begin{equation}\label{eq:rank_r_roots_derivatives_convergence}
\left(
s_{k,N}^{(r)},
\widetilde{\mathcal G}_{N,r}'(s_{k,N}^{(r)})
\right)
\longrightarrow
\left(
s_k^{(r)},
\mathcal G_{r-1}'(s_k^{(r)})
\right)
\end{equation}
almost surely, simultaneously for all \(k\ge1\).

On the other hand, Proposition~\ref{pro:convergence_rank_r} gives,
for every fixed \(K\ge1\),
\[
\bigl(s_{k,N}^{(r)}\bigr)_{k=1}^K
\Longrightarrow
\bigl(
a_{1,k}^{\Theta}
\bigr)_{k=1}^K.
\]
Combining this with
\eqref{eq:rank_r_roots_derivatives_convergence} and uniqueness of the
weak limit yields
\begin{equation}\label{eq:rank_r_roots_equal_Airy_process}
\bigl(s_k^{(r)}\bigr)_{k=1}^K
\overset{d}{=}
\bigl(
a_{1,k}^{\Theta}
\bigr)_{k=1}^K.
\end{equation}
Since the roots \(\{s_k^{(r)}\}_{k\ge1}\) are almost surely strictly
ordered, and since \(K\) is arbitrary,
\eqref{eq:rank_r_roots_equal_Airy_process} proves that the rank-\(r\)
Airy point process is almost surely simple:
\[
a_{1,1}^{\Theta}
>
a_{1,2}^{\Theta}
>
\cdots
\qquad\text{almost surely}.
\]
This completes the induction proof of the simplicity assertion in
Proposition~\ref{pro:existence_rank_r} for \(\beta=1\).

Finally, by \eqref{eq:general-overlap-exact}, for every fixed
\(K\ge1\),
\[
\left(
s_{k,N}^{(r)},
N^{1/3}|\langle e_r,\x_k^{(D)}\rangle|^2
\right)_{k=1}^K
\longrightarrow
\left(
s_k^{(r)},
-\frac{1}{
\mathcal G_{r-1}'(s_k^{(r)})
}
\right)_{k=1}^K
\qquad\text{almost surely}
\]
under the coupling above. Hence the corresponding convergence holds
jointly in distribution. This completes the simultaneous induction
and the proof of the GOE case of
Theorem~\ref{finite_spikes_GOE}.

\begin{proof}[Proof of
Lemma~\ref{lem:joint-edge-gaussian-general-short}]
By Proposition~\ref{pro:convergence_rank_r}, for every fixed
\(M\ge1\),
\begin{equation*} % \label{eq:edge-convergence-r-minus-one}
(a_{m,N})_{m=1}^M
\Longrightarrow
(a_m^{\Theta^{(r)}})_{m=1}^M.
\end{equation*}

Define
\[
A_N
:=
\bigl(
\langle e_i,\widehat\u_{j,N}\rangle
\bigr)_{
1\le i\le r-1,\,
1\le j\le M
}
\in\mathbb R^{(r-1)\times M},
\]
and
\[
B_N
:=
\bigl(
\langle e_i,\widehat\u_{j,N}\rangle
\bigr)_{
r\le i\le N,\,
1\le j\le M
}
\in\mathbb R^{(N-r+1)\times M}.
\]
The choices of signs of the eigenvectors only multiply the columns
of \(A_N\) and \(B_N\) by signs and hence do not affect the squared
coordinates considered below.

Since
\(\widehat\u_{1,N},\dots,\widehat\u_{M,N}\)
are orthonormal,
\[
A_N^\top A_N+B_N^\top B_N=I_M.
\]
Set
\[
C_N:=B_N^\top B_N=I_M-A_N^\top A_N.
\]

The law of \(\widehat H_N\) is invariant under conjugation by matrices
of the form
\[
I_{r-1}\oplus O,
\qquad
O\in\mathrm O(N-r+1).
\]
Consequently, conditionally on the eigenvalues and on the upper block
\(A_N\), the lower block \(B_N\), modulo the irrelevant column signs,
is Haar distributed on the orthogonal orbit determined by
\(B_N^\top B_N=C_N\). Therefore, for the sign-invariant quantities
considered here, its conditional law is the same as that of
\[
Q_NC_N^{1/2},
\]
where \(Q_N\) is Haar distributed on the Stiefel manifold
\[
V_M(\mathbb R^{N-r+1})
=
\left\{
Q\in\mathbb R^{(N-r+1)\times M}:Q^\top Q=I_M
\right\}
\]
and may be taken independent of the conditioning variables.

By the induction hypothesis, for every fixed
\(1\le i\le r-1\) and \(1\le j\le M\), the sequence
\[
N^{1/3}
|\langle e_i,\widehat\u_{j,N}\rangle|^2
\]
is tight. Since \(r\) and \(M\) are fixed,
\[
\|A_N\|_{\mathrm{HS}}^2
=
\sum_{i=1}^{r-1}\sum_{j=1}^M
|\langle e_i,\widehat\u_{j,N}\rangle|^2
=
O_{\mathbb P}(N^{-1/3}).
\]
It follows that
\[
C_N^{1/2}
=
I_M+O_{\mathbb P}(N^{-1/3})
\]
in operator norm.

Let \(q_N\) denote the first row of \(Q_N\). For fixed \(M\),
\[
\sqrt{N-r+1}\,q_N
\Longrightarrow
(\eta_1,\dots,\eta_M),
\]
where \(\eta_1,\dots,\eta_M\) are i.i.d.\ standard Gaussian random
variables; see \cite[Theorem~2.8]{Meckes19}. Since \(Q_N\) is
independent of the conditioning variables, the limiting Gaussian
vector is independent of every subsequential limit of
\[
\bigl(
(a_{m,N})_{m=1}^M,A_N
\bigr).
\]

Combining these observations with \((N-r+1)/N\to1\), we obtain
\eqref{eq:joint-edge-weight-convergence}
which proves the lemma.
\end{proof}

\subsection{Sketch of proof of the GUE case of
Theorem~\ref{finite_spikes_GOE}}
\label{subsec:proof_GUE}

The proof is the complex analogue of the GOE argument. We prove
simultaneously the overlap convergence and the almost-sure simplicity
of the higher-rank Airy\(_2\) point process
\(\{a_{2,k}^{\Theta}\}_{k\ge1}\) ($\Theta = (\theta_1,\dots,\theta_r)$).

For \(r=1\), this point process is the negative spectrum of the scalar
stochastic Airy operator with Robin boundary condition
\(f'(0)=\theta_1f(0)\). Its almost-sure simplicity follows directly
from \cite[Lemma~2.7]{Bloemendal-Virag11}. The rank-one overlap
convergence follows from the \(\beta=2\) case of
Theorem~\ref{one_spike_beta}. This establishes the induction base.

Assume now that the rank-\((r-1)\) Airy\(_2\) point process is almost
surely simple. By unitary invariance and relabeling the spikes, we may
assume that \(\w_j=e_j\), \(1\le j\le r\), and \(i=r\). Set
\(\Theta^{(r)}=(\theta_1,\dots,\theta_{r-1})\) and
\(\mathcal G_{r-1}^{(2)}
:=\mathcal G^{(2)}_{(r-1),r^c}\). Thus,
\[
a_{2,1}^{\Theta^{(r)}}
>
a_{2,2}^{\Theta^{(r)}}
>
\cdots
\qquad\text{almost surely}.
\]

Define
\[
\widehat H_N
:=
H_N+\sum_{j=1}^{r-1}d_je_je_j^\ast,
\qquad
H_N^{(D)}
=
\widehat H_N+d_re_re_r^\ast.
\]
Let \(\widehat\mu_{1,N}\ge\cdots\ge\widehat\mu_{N,N}\) be the
eigenvalues of \(\widehat H_N\), with corresponding normalized
eigenvectors
\(\widehat\u_{1,N},\dots,\widehat\u_{N,N}\), and put
\(a_{2,m,N}:=N^{2/3}(\widehat\mu_{m,N}-2)\). Define
\[
\begin{aligned}
\widetilde{\mathcal G}^{(2)}_{N,r}(w)
&:=
N^{1/3}
\left[
e_r^\ast
\bigl((2+N^{-2/3}w)I-\widehat H_N\bigr)^{-1}
e_r-1
\right] \\
&=
\sum_{m=1}^N
\frac{
N|\langle e_r,\widehat\u_{m,N}\rangle|^2
}{
w-a_{2,m,N}
}
-N^{1/3}.
\end{aligned}
\]

Set \(s_{2,k,N}^{(r)}:=N^{2/3}(\lambda_k^{(D)}-2)\). The rank-one
determinant identity and Lemma~\ref{lem:eigenvector_eigenvalue}, with
transpose replaced by adjoint, give
\begin{equation}\label{eq:GUE_rank_r_root_overlap}
\widetilde{\mathcal G}^{(2)}_{N,r}(s_{2,k,N}^{(r)})
=
\frac{\theta_r}{d_r},
\qquad
|\langle e_r,\x_k^{(D)}\rangle|^2
=
-\frac{N^{-1/3}}{
d_r^2
(\widetilde{\mathcal G}^{(2)}_{N,r})'
(s_{2,k,N}^{(r)})
}.
\end{equation}

The proof of
Lemma~\ref{lem:joint-edge-gaussian-general-short} carries over to the
complex setting after replacing
\(\mathbb R\), \(\mathrm O(\cdot)\), and \((\cdot)^\top\) by
\(\mathbb C\), \(\mathrm U(\cdot)\), and \((\cdot)^\ast\),
respectively. In particular, for every fixed \(M\ge1\),
\[
\left(
(a_{2,m,N})_{m=1}^M,\,
\bigl(
N|\langle e_r,\widehat\u_{m,N}\rangle|^2
\bigr)_{m=1}^M
\right)
\Longrightarrow
\left(
(a_{2,m}^{\Theta^{(r)}})_{m=1}^M,\,
(|\eta_m|^2)_{m=1}^M
\right),
\]
where \(\eta_1,\dots,\eta_M\) are i.i.d.\ standard complex Gaussian
random variables, independent of
\(\{a_{2,m}^{\Theta^{(r)}}\}_{m\ge1}\). Here
\(|\eta_m|^2\overset{d}{=}\frac12\chi_2^2\).

Indeed, the conditional Haar-orbit argument in the proof of
Lemma~\ref{lem:joint-edge-gaussian-general-short} is unchanged, with
the real Stiefel manifold replaced by its complex counterpart. The
necessary second and fourth moment estimates follow from
Lemma~\ref{lem:spherical-moments-subspace} with
\(\beta_{\mathbb F}=2\).

The proof of Proposition~\ref{prop:general-meromorphic} therefore also
carries over to the complex setting and yields a coupling under which
\begin{equation}\label{eq:GUE_general_meromorphic_convergence}
\widetilde{\mathcal G}^{(2)}_{N,r}
\xrightarrow{\mathrm{mer}}
\mathcal G_{r-1}^{(2)}
\qquad\text{almost surely}.
\end{equation}
The limiting weights are precisely
\(\frac12\chi_2^2\), so the limiting function is
\(\mathcal G^{(2)}_{(r-1),r^c}\).

By the induction hypothesis, the poles
\(\{a_{2,k}^{\Theta^{(r)}}\}_{k\ge1}\) are almost surely strictly
ordered. The same argument as in the GOE case shows that, for every \(k\ge1\), the equation
\(\mathcal G_{r-1}^{(2)}(x)=\theta_r\) has a unique solution
\[
s_{2,k}^{(r)}
\in
\bigl(
a_{2,k}^{\Theta^{(r)}},
a_{2,k-1}^{\Theta^{(r)}}
\bigr),
\qquad
a_{2,0}^{\Theta^{(r)}}:=+\infty.
\]
In particular,
\[
s_{2,1}^{(r)}
>
s_{2,2}^{(r)}
>
\cdots
\qquad\text{almost surely}.
\]

Since \(\theta_r/d_r\to\theta_r\),
\eqref{eq:GUE_general_meromorphic_convergence} and
Lemma~\ref{lem:simple-zero-stability-meromorphic} imply
\begin{equation}\label{eq:GUE_rank_r_root_derivative_convergence}
\left(
s_{2,k,N}^{(r)},
(\widetilde{\mathcal G}^{(2)}_{N,r})'
(s_{2,k,N}^{(r)})
\right)
\longrightarrow
\left(
s_{2,k}^{(r)},
(\mathcal G_{r-1}^{(2)})'
(s_{2,k}^{(r)})
\right)
\end{equation}
almost surely under the corresponding coupling, simultaneously for
all \(k\ge1\).

For every fixed \(K\ge1\),
Proposition~\ref{pro:convergence_rank_r} gives
\[
\bigl(s_{2,k,N}^{(r)}\bigr)_{k=1}^K
\Longrightarrow
\bigl(
a_{2,k}^{\Theta}
\bigr)_{k=1}^K.
\]
Combining this convergence with
\eqref{eq:GUE_rank_r_root_derivative_convergence} and uniqueness of
the weak limit yields
\begin{equation*} % \label{eq:GUE_rank_r_roots_equal_Airy_process}
\bigl(s_{2,k}^{(r)}\bigr)_{k=1}^K
\overset{d}{=}
\bigl(
a_{2,k}^{\Theta}
\bigr)_{k=1}^K.
\end{equation*}
Since the roots on the left are almost surely strictly ordered and
\(K\) is arbitrary, this proves the almost-sure simplicity of the
rank-\(r\) Airy\(_2\) point process:
\[
a_{2,1}^{\Theta}
>
a_{2,2}^{\Theta}
>
\cdots
\qquad\text{almost surely}.
\]
This establishes the simplicity assertion in
Proposition~\ref{pro:existence_rank_r} for \(\beta=2\).

Finally, substituting
\eqref{eq:GUE_rank_r_root_derivative_convergence} into the exact
overlap identity \eqref{eq:GUE_rank_r_root_overlap} gives, for every
fixed \(K\ge1\),
\[
\left(
s_{2,k,N}^{(r)},
N^{1/3}|\langle e_r,\x_k^{(D)}\rangle|^2
\right)_{k=1}^K
\longrightarrow
\left(
s_{2,k}^{(r)},
-\frac{1}{
(\mathcal G_{r-1}^{(2)})'(s_{2,k}^{(r)})
}
\right)_{k=1}^K
\qquad\text{almost surely}
\]
under the coupling above. Hence the corresponding convergence holds
jointly in distribution. This completes the simultaneous induction
and the proof of the GUE case of
Theorem~\ref{finite_spikes_GOE}.

\section{Proof of Proposition~\ref{prop:one-coordinate-BBP}}
% \label{sec:critical-overlap-BBP}

In this section, we use the notation of
Proposition~\ref{prop:one-coordinate-BBP}. When
\(\beta\notin\{1,2\}\), we have \(r=1\) and
\(\Theta^{(i)}=()\), and the symbols
\[
\theta_i,\quad
a_{\beta,k}^{\Theta^{(i)}},\quad
s_{\beta,k}^{(i)},\quad
\mathcal G^{(\beta)}_{(r-1),i^c},\quad
\xi_{\beta,k}^2
\]
are understood as
\[
\theta,\quad
a_{\beta,k},\quad
s_{\beta,k},\quad
\mathcal G^{(\beta)},\quad
\xi_{\beta,k}^2,
\]
respectively, as in Theorem~\ref{one_spike_beta}.

Near a pole \(a_{\beta,m}^{\Theta^{(i)}}\),
Proposition~\ref{pro:multispike_airy_green} for
\(\beta\in\{1,2\}\), and
Proposition~\ref{prop:beta_airy_green} for general \(\beta>0\), give
\begin{align}
\mathcal G^{(\beta)}_{(r-1),i^c}(x)
&=
\frac{\beta^{-1}\xi_{\beta,m}^2}
{x-a_{\beta,m}^{\Theta^{(i)}}}
+O(1),
\label{eq:pole-expansion-direct}\\
\bigl(\mathcal G^{(\beta)}_{(r-1),i^c}\bigr)'(x)
&=
-\frac{\beta^{-1}\xi_{\beta,m}^2}
{\bigl(x-a_{\beta,m}^{\Theta^{(i)}}\bigr)^2}
+O(1).
\label{eq:pole-derivative-direct}
\end{align}
Moreover,
\(\bigl(\mathcal G^{(\beta)}_{(r-1),i^c}\bigr)'(x)<0\)
away from the poles, so
\(\mathcal G^{(\beta)}_{(r-1),i^c}\) is strictly decreasing on every
interval between consecutive poles.

As \(\theta_i\to+\infty\), monotonicity gives
\[
s_{\beta,k}^{(i)}
\downarrow
a_{\beta,k}^{\Theta^{(i)}}.
\]
Taking \(m=k\) in \eqref{eq:pole-expansion-direct}, we obtain
\[
\theta_i
=
\frac{\beta^{-1}\xi_{\beta,k}^2}
{s_{\beta,k}^{(i)}-a_{\beta,k}^{\Theta^{(i)}}}
+O(1),
\]
which proves the first asymptotic relation in
\eqref{eq:positive-direct}. Substitution into
\eqref{eq:pole-derivative-direct} then proves the second asymptotic
relation in \eqref{eq:positive-direct}.

If \(\theta_i\to-\infty\) and \(k\ge2\), then
\[
s_{\beta,k}^{(i)}
\uparrow
a_{\beta,k-1}^{\Theta^{(i)}}.
\]
Taking \(m=k-1\) in \eqref{eq:pole-expansion-direct} gives
\[
\theta_i
=
\frac{\beta^{-1}\xi_{\beta,k-1}^2}
{s_{\beta,k}^{(i)}-a_{\beta,k-1}^{\Theta^{(i)}}}
+O(1).
\]
This proves the first asymptotic relation in
\eqref{eq:negative-kge2-direct}; the second asymptotic relation follows
from \eqref{eq:pole-derivative-direct}.

It remains to consider \(k=1\) as \(\theta_i\to-\infty\).
Monotonicity gives
\[
s_{\beta,1}^{(i)}\longrightarrow+\infty.
\]
By the defining root equations \eqref{eq:roots_beta_12} and
\eqref{eq:roots_beta_positive}, together with
Part~\ref{enu:lem:beta_airy_green_infty:1} of
Lemma~\ref{lem:beta_airy_green_infty},
\[
\theta_i
=
\mathcal G^{(\beta)}_{(r-1),i^c}
\bigl(s_{\beta,1}^{(i)}\bigr)
=
-\sqrt{s_{\beta,1}^{(i)}}+o(1),
\]
which proves the first asymptotic relation in
\eqref{eq:negative-k1-direct}.

The same lemma gives, uniformly in the right half-plane,
\[
\mathcal G^{(\beta)}_{(r-1),i^c}(z)+\sqrt z=o(1),
\qquad |z|\to\infty.
\]
Cauchy's formula on \(|z-x|=x/2\) therefore yields
\[
\bigl(\mathcal G^{(\beta)}_{(r-1),i^c}\bigr)'(x)
=
-\frac1{2\sqrt x}+o(x^{-1}),
\qquad x\to+\infty.
\]
Evaluating this at \(x=s_{\beta,1}^{(i)}\) and using the first
asymptotic relation in \eqref{eq:negative-k1-direct}, we obtain
\[
\bigl(\mathcal G^{(\beta)}_{(r-1),i^c}\bigr)'
\bigl(s_{\beta,1}^{(i)}\bigr)
=
-\frac1{2|\theta_i|}
+o(|\theta_i|^{-2}).
\]
Consequently,
\[
-\frac1{
\bigl(\mathcal G^{(\beta)}_{(r-1),i^c}\bigr)'
\bigl(s_{\beta,1}^{(i)}\bigr)}
=
2|\theta_i|+o(1),
\]
which proves the second asymptotic relation in
\eqref{eq:negative-k1-direct}.

We now assume \(\beta\in\{1,2\}\), since the remaining argument uses
Part~\ref{enu:lem:beta_airy_green_infty:2} of
Lemma~\ref{lem:beta_airy_green_infty}. For notational convenience
within this part only, set
\[
\begin{gathered}
t:=-\theta_i,\qquad
\mathcal G:=\mathcal G^{(\beta)}_{(r-1),i^c},
\qquad
s_t:=s_{\beta,1}^{(i)},\\
U_t:=t^{1/2}\bigl(\mathcal G(t^2)+t\bigr),
\qquad
V_t:=t^{1/2}
\left(
-\frac1{\mathcal G'(t^2)}-2t
\right).
\end{gathered}
\]
Applying Part~\ref{enu:lem:beta_airy_green_infty:2} of Lemma~\ref{lem:beta_airy_green_infty} with \(x=t^2\), we
obtain
\begin{equation}\label{eq:joint-G-reciprocal-at-t2}
\begin{pmatrix}
U_t\\ V_t
\end{pmatrix}
\Longrightarrow
\mathcal N\!\left(
\begin{pmatrix}0\\0\end{pmatrix},
\frac1\beta
\begin{pmatrix}
1&-1\\
-1&2
\end{pmatrix}
\right).
\end{equation}

Since \(\mathcal G(s_t)=-t\) and
\(\mathcal G(s_t)=-\sqrt{s_t}+o(1)\), we have
\(\sqrt{s_t}=t+o(1)\), and hence
\[
s_t=t^2+o(t)
\qquad\text{almost surely}.
\]
By the mean value theorem, for some random point \(\zeta_t\) between
\(t^2\) and \(s_t\),
\[
0
=
\mathcal G(s_t)+t
=
\mathcal G(t^2)+t
+\mathcal G'(\zeta_t)(s_t-t^2).
\]
Since \(\zeta_t=t^2+o(t)\), the derivative asymptotics in
Lemma~\ref{lem:beta_airy_green_infty} give
\[
\mathcal G'(\zeta_t)
=
-\frac1{2t}\bigl(1+o(1)\bigr)
\qquad\text{almost surely}.
\]
Together with \(U_t=O_{\mathbb P}(1)\), this yields
\begin{equation}\label{eq:supercritical-root-expansion-short}
\frac{s_t-t^2}{t^{1/2}}
=
2U_t+o_{\mathbb P}(1).
\end{equation}
Consequently,
\[
\frac{s_t-t^2}{2t^{1/2}}
\Longrightarrow
\mathcal N\!\left(0,\frac1\beta\right),
\]
which proves the first convergence in
\eqref{eq:normal_supercritical}.

For the overlap fluctuation, set
\[
R(x):=-\frac1{\mathcal G'(x)}.
\]
The derivative asymptotics in
Lemma~\ref{lem:beta_airy_green_infty} imply
\[
R'(x)
=
\frac{\mathcal G''(x)}{\mathcal G'(x)^2}
=
\frac1{\sqrt x}\bigl(1+o(1)\bigr)
\qquad\text{almost surely}.
\]
Applying the mean value theorem once more, for some random point
\(\eta_t\) between \(t^2\) and \(s_t\),
\[
R(s_t)-R(t^2)
=
R'(\eta_t)(s_t-t^2).
\]
Since \(\eta_t=t^2+o(t)\), it follows from
\eqref{eq:supercritical-root-expansion-short} that
\[
t^{1/2}\bigl(R(s_t)-R(t^2)\bigr)
=
2U_t+o_{\mathbb P}(1).
\]
Therefore,
\begin{align*}
t^{1/2}
\left[
-\frac1{\mathcal G'(s_t)}-2t
\right]
&=
t^{1/2}\bigl(R(s_t)-R(t^2)\bigr)
+t^{1/2}\bigl(R(t^2)-2t\bigr)\\
&=
2U_t+V_t+o_{\mathbb P}(1).
\end{align*}
If \((U,V)\) denotes the Gaussian limit in
\eqref{eq:joint-G-reciprocal-at-t2}, then
\[
\operatorname{Var}(2U+V)
=
\frac4\beta+\frac2\beta-\frac4\beta
=
\frac2\beta.
\]
It follows that
\[
t^{1/2}
\left[
-\frac1{\mathcal G'(s_t)}-2t
\right]
\Longrightarrow
\mathcal N\!\left(0,\frac2\beta\right),
\]
which proves the second convergence in
\eqref{eq:normal_supercritical} and completes the proof of
Proposition~\ref{prop:one-coordinate-BBP}.
\section{Properties of Airy-Green functions}
\label{sec:supplemental}

In this section we prove Propositions~\ref{prop:beta_airy_green}, \ref{pro:multispike_airy_green} and Lemma \ref{lem:beta_airy_green_infty} in Section \ref{s.preliminaries}, and prove Propositions \ref{enu:prop:beta_airy_green_2} and \ref{prop:general-meromorphic} that are used in the proofs of Theorems \ref{finite_spikes_GOE} and \ref{one_spike_beta}.

\subsection{Proofs of Proposition~\ref{prop:beta_airy_green} and Lemma \ref{lem:beta_airy_green_infty} for $\mathcal G^{(\beta)}$}

The proof of Proposition~\ref{prop:beta_airy_green} extends the
argument of \cite[Theorem~8.1]{Bykhovskaya-Gorin-Sodin25} from
\(\beta=1\) to general \(\beta>0\). For
Lemma~\ref{lem:beta_airy_green_infty}, the proof of
Part~\ref{enu:lem:beta_airy_green_infty:1} is adapted from the
argument of \cite[Proposition~8.3]{Bykhovskaya-Gorin-Sodin25}, whereas
the fluctuation and derivative estimates in
Part~\ref{enu:lem:beta_airy_green_infty:2} require additional
arguments.

\subsubsection{Rigidity input for the Airy\texorpdfstring{$_\beta$}{beta} point process}

In the proof of Proposition~\ref{prop:beta_airy_green}, we need the rigidity result of the Airy\texorpdfstring{$_\beta$}{beta} point process that is a generalization of \cite[Lemma 8.9]{Bykhovskaya-Gorin-Sodin25}.

\begin{lem}\label{lem:airy-bound}
  For every \(\varepsilon>0\), there exists a random variable \(\mathcal J=\mathcal J(\varepsilon)\) such that, almost surely, for all \(j>\mathcal J\),
\begin{equation}\label{eq:est_ap_k}
  \left|a_{\beta,j}+\left(\frac{3\pi j}{2}\right)^{2/3}\right| \le j^{\frac16+\varepsilon}.
\end{equation}
\end{lem}

The $\beta=1$ case of Lemma~\ref{lem:airy-bound} is stated in
\cite[Lemma~8.9]{Bykhovskaya-Gorin-Sodin25}, where
$a_{\beta,j}$ is denoted by $\mathfrak{a}_j$, and the exponent
$1/6+\epsilon$ can be strengthened to $\epsilon$. For $\beta=2$,
the corresponding stronger estimate follows from the mean asymptotics
stated immediately before \cite[Theorem~1]{Soshnikov00b} and the
variance estimate in that theorem. In
\cite[Section~8]{Bykhovskaya-Gorin-Sodin25}, the $\beta=1$ estimate
is obtained using the corresponding mean asymptotics together with a
variance estimate deduced from the $\beta=2$ case via the
Forrester--Rains relation \cite[Theorem~4.3]{Forrester-Rains01};
see also \cite{ORourke10}. Here we instead use
\cite[Theorem~1 and Corollary~2]{Ashbury_Bridgwood22}; see also
\cite[Proposition~3.2]{Zhong19}.
\begin{proof}
From \cite[Theorem~1 and Corollary~2]{Ashbury_Bridgwood22}, we have that for every \(\varepsilon\in(0,1/6)\), there exist \(k_0=k_0(\varepsilon)\), \(\chi=\chi(\varepsilon)>0\), and \(C=C(\varepsilon)>0\) such that for all \(k\ge k_0\),
\[
\mathbb P\!\left(
\left|a_{\beta,k}+\left(\frac{3\pi k}{2}\right)^{2/3}\right|
\ge Ck^{\frac16+\varepsilon}
\right)
\le e^{-k^\chi}.
\]
Then the Borel--Cantelli lemma yields the claim.
\end{proof}

\subsubsection{Proof of  Proposition~\ref{prop:beta_airy_green}} % \label{subsubsec:proof_G_beta}

Analogous to the proof of \cite[Theorem 8.1]{Bykhovskaya-Gorin-Sodin25}, we define, like the decomposition in \cite[Equations (8.19), (8.20) and (8.24)]{Bykhovskaya-Gorin-Sodin25},
\begin{align}
  \mathcal G_1^{(\beta),*}(w) &:= \lim_{x\to-\infty} \Biggl[ \sum_{j:\,a_{\beta,j}>x} b_j(w) \Biggr], \quad \text{where} \quad b_j(w):= \frac{i-w}{(w-a_{\beta,j})(i-a_{\beta,j})}, \label{eq:G_beta1_firsta} \\
\mathcal G_1^{(\beta),**, \Im} &:= -\lim_{x\to-\infty} \sum_{j:\,a_{\beta,j}>x}\frac{1}{1+a_{\beta,j}^2}, \label{eq:G_beta1_seconda} \\
\mathcal G_1^{(\beta),**, \Re} &:= \lim_{x\to-\infty} \Biggl[ \sum_{j:\,a_{\beta,j}>x}\frac{-a_{\beta,j}}{1+a_{\beta,j}^2} -\frac{2}{\pi}\sqrt{-x} \Biggr], \label{eq:G_beta1_first} \\
  \mathcal G_2^{(\beta),*}(w) &:= \lim_{x\to-\infty} \Biggl[ \sum_{j:\,a_{\beta,j}>x} \eta_j b_j(w) \Biggr], \quad \text{where} \quad \eta_j:=\beta^{-1}\xi_{\beta,j}^2-1,  \label{eq:G_beta2_first} \\
\mathcal G_2^{(\beta),**, \Im} &:= -\lim_{x\to-\infty} \sum_{j:\,a_{\beta,j}>x} \frac{\eta_j}{1+a_{\beta,j}^2}, \label{eq:Gbeta2_seconda} \\
\mathcal G_2^{(\beta),**, \Re} &:= \lim_{x\to-\infty} \sum_{j:\,a_{\beta,j}>x} \eta_j \frac{-a_{\beta,j}}{1+a_{\beta,j}^2}. \label{eq:G_beta2_firstc}
\end{align}
We also set
\[
\mathcal G_{\star}^{(\beta)}(w) := \mathcal G_{\star}^{(\beta),*}(w) + i\mathcal G_{\star}^{(\beta),**, \Im} + \mathcal G_{\star}^{(\beta),**, \Re}, \quad \star = 1 \text{ or } 2.
\]
Then \( \mathcal G^{(\beta)}(w) = \mathcal G_1^{(\beta)}(w) + \mathcal G_2^{(\beta)}(w) \).

The $\mathcal G_1^{(\beta)}(w)$ part is handled like the $\mathcal G_1(w)$ term defined in \cite[Equation (8.20)]{Bykhovskaya-Gorin-Sodin25}, with \cite[Lemma 8.9]{Bykhovskaya-Gorin-Sodin25} replaced by Lemma~\ref{lem:airy-bound} in our case. Indeed, by the estimate of $a_{\beta, j}$ in \eqref{eq:est_ap_k}, we have that for \( w  \) in a compact subset of \( \mathbb C\setminus\{a_{\beta,j}:j\ge1\}\), uniformly
\begin{equation*} % \label{eq:est_bj(w)}
  \left| b_j(w) \right| = \bigO(j^{-4/3})
\end{equation*}
as $j \to \infty$, and similarly \((1+a_{\beta,j}^2)^{-1}= \bigO(j^{-4/3})\) as $j \to \infty$. Therefore, \(\mathcal G_1^{(\beta),*}(w)\) is well defined and analytic on \(\mathbb C\setminus\{a_{\beta,j}:j\ge1\}\), and the convergence of \eqref{eq:G_beta1_firsta} is uniform in a compact subset of \( \mathbb C\setminus\{a_{\beta,j}:j\ge1\}\). Similarly, the right-hand side of \eqref{eq:G_beta1_seconda} converges absolutely, and the convergence of the right-hand side of \eqref{eq:G_beta1_first} is obtained by the  approximating $a_{\beta, j}$ by \(-(3\pi j/2)^{2/3}\) as in \cite[Equations (8.21) and (8.22)]{Bykhovskaya-Gorin-Sodin25}. Here we note that in \cite[Equation (8.22)]{Bykhovskaya-Gorin-Sodin25}, the error of the approximation is controlled as $\mathfrak{a}_j + (3\pi j/2)^{2/3} = \bigO(j^{\epsilon})$, but it suffices to use the error bound $\bigO(j^{1/6 + \epsilon})$ as in \eqref{eq:est_ap_k}. Hence \(\mathcal G_1^{(\beta),**, \Im}\) and \(\mathcal G_1^{(\beta),**, \Re}\) are well defined. We omit the details.

The $\mathcal G_2^{(\beta)}(w)$ part is handled like the
$\mathcal G_2(w)$ term defined in
\cite[Equation (8.24)]{Bykhovskaya-Gorin-Sodin25}.
Condition on a realization of $\{a_{\beta,j}\}$ satisfying
\eqref{eq:est_ap_k} for all sufficiently large $j$. By
\eqref{eq:est_ap_k},
\[
\sum_{j=1}^{\infty}
\frac{1}{(1+a_{\beta,j}^2)^2}<\infty,
\qquad
\sum_{j=1}^{\infty}
\frac{a_{\beta,j}^2}{(1+a_{\beta,j}^2)^2}<\infty.
\]
Since the random variables $\eta_j$ are independent, centered, and
have finite variance, Kolmogorov's two-series theorem implies that
the series on the right-hand sides of
\eqref{eq:Gbeta2_seconda} and \eqref{eq:G_beta2_firstc}
converge almost surely. Hence
$\mathcal G_2^{(\beta),**, \Im}$ and
$\mathcal G_2^{(\beta),**, \Re}$ are well defined almost surely.

It remains to consider $\mathcal G_2^{(\beta),*}(w)$. Since
$\mathbb E|\eta_1|<\infty$ and
$(1+a_{\beta,j}^2)^{-1}=\bigO(j^{-4/3})$, Tonelli's theorem gives,
almost surely,
\begin{equation} \label{eq:conv_G_2Im}
  \sum_{j=1}^{\infty}
  \frac{|\eta_j|}{1+a_{\beta,j}^2}<\infty.
\end{equation}
Let
\[
K\Subset\mathbb C\setminus\{a_{\beta,j}:j\geq1\}.
\]
By the definition of $b_j(w)$ and \eqref{eq:est_ap_k}, for all
sufficiently large $j$,
\[
\sup_{w\in K}|b_j(w)|
\leq
\frac{C_K}{1+a_{\beta,j}^2}.
\]
Consequently, by \eqref{eq:conv_G_2Im}, almost surely,
\begin{equation*} % \label{eq:condition_Omega_-}
  \sum_{j=1}^{\infty}
  |\eta_j|\sup_{w\in K}|b_j(w)|<\infty.
\end{equation*}
The Weierstrass $M$-test therefore implies that the series on the
right-hand side of \eqref{eq:G_beta2_first} converges absolutely and
uniformly on $K$, almost surely. Since $K$ is arbitrary,
$\mathcal G_2^{(\beta),*}$ is well defined and analytic on
$\mathbb C\setminus\{a_{\beta,j}:j\geq1\}$.

This proves Proposition~\ref{prop:beta_airy_green}.
\begin{remark} \label{rmk:alt_G_beta}
  We remark here that based on the convergences proved above, the series
  \begin{equation*}
    \sum^{\infty}_{j = 1} \beta^{-1} \xi^2_{\beta, j} \left( \frac{1}{w - a_{\beta, j}} + \frac{a_{\beta, j}}{1 + a^2_{\beta, j}} \right)
  \end{equation*}
  converges and is equal to $\mathcal G_1^{(\beta),*}(w) + i \mathcal G_1^{(\beta),**, \Im} + \mathcal G_2^{(\beta),*}(w) + i \mathcal G_2^{(\beta),**, \Im}$. This result will be used in the proof of Proposition~\ref{enu:prop:beta_airy_green_2}.
\end{remark}

\subsubsection{Proof of Lemma~\ref{lem:beta_airy_green_infty}} for \(\mathcal G^{(\beta)}\).

The proof is similar to the proof of \cite[Proposition~8.3]{Bykhovskaya-Gorin-Sodin25}. We write, as in \cite[Equation (8.25)]{Bykhovskaya-Gorin-Sodin25},
\begin{equation*} % \label{squre w}
    \mathcal G^{(\beta)}(w)
=
\mathcal G_1^{***}(w)
+\mathcal G_2^{(\beta),***}(w)
+\mathcal G_3^{(\beta),***}(w),
\end{equation*}
where, with $\eta_j$ defined in \eqref{eq:G_beta2_first},
\begin{align*}
\mathcal G_1^{***}(w) &= \lim_{x\to-\infty} \Biggl[ \sum_{j:\,-(\frac{3\pi j}{2})^{2/3}>x} \frac{1}{w+(\frac{3\pi j}{2})^{2/3}} -\frac{2}{\pi}\sqrt{-x} \Biggr], \\
\mathcal G_2^{(\beta),***}(w) &= \sum^{\infty}_{j = 1} \frac{\eta_j}{w+(\frac{3\pi j}{2})^{2/3}}, \\
\mathcal G_3^{(\beta),***}(w) &= \sum^{\infty}_{j = 1} \beta^{-1}\xi_{\beta,j}^2\, \frac{(\frac{3\pi j}{2})^{2/3}+a_{\beta,j}}{(w-a_{\beta,j})(w+(\frac{3\pi j}{2})^{2/3})}.
\end{align*}
We note that $\mathcal G_1^{***}(w)$ is independent of $\beta$, and it is the first term on the right-hand side of \cite[Equation (8.25)]{Bykhovskaya-Gorin-Sodin25}. \cite[Equation (8.29)]{Bykhovskaya-Gorin-Sodin25} yields

\[
\mathcal G_1^{***}(w)+\sqrt w = \bigO(\lvert w \rvert^{-\frac{1}{2}})
\qquad\text{as } |w|\to\infty,\ \Re w\ge0.
\]
The estimates of $\mathcal G_2^{(\beta),***}(w)$ and $\mathcal G_3^{(\beta),***}(w)$ are similar to the estimates of the second and third terms on the right-hand side of \cite[Equation (8.25)]{Bykhovskaya-Gorin-Sodin25}. We give the key steps of the estimates below and omit some details.

We first show that \(\mathcal G_3^{(\beta),***}(w) = \bigO(\lvert w \rvert^{-1/4 + c}) \) for any $c > 0$, in the sense that there exists a random variable $A = A(c)$, such that for all $w$ with $\Re(w) \geq 0$ and $\lvert w \rvert \geq 1$, $\lvert \mathcal G_3^{(\beta),***}(w) \rvert \leq A \lvert w \rvert^{-1/4 + c}$. To see it, we apply the Borel--Cantelli lemma to  $\beta^{-1}\xi_{\beta,j}^2$ and apply Lemma \ref{lem:airy-bound} to $(3\pi j/2)^{2/3} + a_{\beta,j}$, and find that for every small \(\epsilon>0\), almost surely there exists \(C_\epsilon>0\) such that with $\mathcal{J} = \mathcal{J}(\epsilon)$ defined in Lemma \ref{lem:airy-bound},
\begin{align*}
  \beta^{-1}\xi_{\beta,j}^2 \le {}& C_{\epsilon} j^\epsilon, & \left\lvert (\frac{3\pi j}{2})^{2/3}+a_{\beta,j} \right\rvert \leq {}& j^{\frac{1}{6} + \epsilon}, & \text{if } j \ge {}& \mathcal J.
\end{align*}
Then, arguing as in
\cite[Equation~(8.27)]{Bykhovskaya-Gorin-Sodin25} and increasing
\(\mathcal J\) if necessary, we obtain, with
\(M=\lfloor |w|^{3/2}\rfloor\),
\[
\begin{aligned}
\left|\mathcal G_3^{(\beta),***}(w)\right|
\le {}&
\sum_{j=1}^{\mathcal J-1}
\beta^{-1}\xi_{\beta,j}^2
\frac{\left|(\frac{3\pi j}{2})^{2/3}+a_{\beta,j}\right|}
{\left|w-a_{\beta,j}\right|
 \left|w+(\frac{3\pi j}{2})^{2/3}\right|} \\
&+
C_\epsilon\sum_{j=\mathcal J}^{M}
\frac{j^{1/6+2\epsilon}}
{\left|w+(\frac{3\pi j}{2})^{2/3}\right|^2}
+
C_\epsilon\sum_{j=M+1}^{\infty}
\frac{j^{1/6+2\epsilon}}
{\left|w+(\frac{3\pi j}{2})^{2/3}\right|^2}.
\end{aligned}
\]
As \(|w|\to\infty\) with \(\Re w\ge0\), the first term is
\(O(|w|^{-2})\) almost surely, whereas
\[
|w|^{-2}\sum_{j\le M}j^{1/6+2\epsilon}
=O\!\left(|w|^{-1/4+3\epsilon}\right)
\]
and
\[
\sum_{j>M}j^{-7/6+2\epsilon}
=O\!\left(|w|^{-1/4+3\epsilon}\right).
\]
(We note that, in the setting of
\cite{Bykhovskaya-Gorin-Sodin25},
\cite[Equation~(8.17)]{Bykhovskaya-Gorin-Sodin25} yields
\[
\left|
\left(\frac{3\pi j}{2}\right)^{2/3}+a_j
\right|
\le j^\epsilon,
\]
without the additional factor \(j^{1/6}\) appearing in
\eqref{eq:est_ap_k}. Consequently, the second and third sums in
\cite[Equation~(8.26)]{Bykhovskaya-Gorin-Sodin25} are both bounded by
\[
O\left(
|w|^{\frac32(2\epsilon-\frac13)}
\right)
=
O\left(|w|^{-1/2+3\epsilon}\right).
\]
Thus, the left-hand side of
\cite[Equation~(8.25)]{Bykhovskaya-Gorin-Sodin25}
in fact admits the sharper bound
\(O(|w|^{-1/2+c})\) for every \(c>0\).)

For \(\mathcal G_2^{(\beta),***}(w)\), we note that \cite{Bykhovskaya-Gorin-Sodin25} proves with the help of the law of the iterated logarithm that for any independent and identically distributed random variables $\{ A_j \}^{\infty}_{j = 1}$ whose mean is $0$ and variance finite, there exists a random variable $A'$ such that for any integer $M$,
\[
\left| \sum_{j=1}^M \frac{A_j}{w + (\frac{3\pi j}{2})^{2/3}} \right| \le \frac{A'}{\lvert w \rvert +M^{2/3}} \sqrt{M\log\log(M+2)} + \sum_{j=1}^\infty \frac{A' \sqrt{j\log\log(j+2)}\,j^{-1/3}}{(\lvert w \rvert + j^{2/3})^2},
\]
and it tends to $0$ as $\lvert w \rvert \to \infty$, uniformly in \(M\). (In \cite{Bykhovskaya-Gorin-Sodin25}, this result is proved with $A_j = \xi^2_j - 1$ with $\{ \xi^2_j \}^{\infty}_{j = 1}$ are independent $\chi^2$ random variables, but the proof depends only on the zero mean and finite variance properties.) Since $\{ \eta_j \}$ satisfies the condition for $\{ A_j \}$, we find \(\mathcal G_2^{(\beta),***}(w)\) tends to $0$ as $\lvert w \rvert \to \infty$, uniformly in \(M\).

Thus the proof of Lemma \ref{lem:beta_airy_green_infty} for \(\mathcal G^{(\beta)}\) is complete.

\begin{remark}
  The proof of Lemma \ref{lem:beta_airy_green_infty} above does not extend \cite[Equation (8.3)]{Bykhovskaya-Gorin-Sodin25} to general $\beta > 0$, because our estimate of \( \mathcal G_3^{(\beta),***}(w) \) is weaker than the desired $o(\lvert w \rvert^{-1/4})$, which is the result that our Lemma \ref{lem:airy-bound} is weaker than \cite[Lemma 8.9]{Bykhovskaya-Gorin-Sodin25} for $\beta = 1$ case.
\end{remark}

\subsection{Proof of Proposition \ref{enu:prop:beta_airy_green_2}} \label{subsec:proof_enu:prop:beta_airy_green_2}

Our proof of  Proposition~\ref{enu:prop:beta_airy_green_2} follows the same strategy of the proof of \cite[Theorem 8.20]{Bykhovskaya-Gorin-Sodin25} that is based on \cite[Theorem~8.13]{Bykhovskaya-Gorin-Sodin25}.

First, we verify that $\mathcal G_N^{(\beta)}(w)$ and $\mathcal G^{(\beta)}(w)$ are all in the function space $\Omega_-$ defined in \cite[Definition 8.10]{Bykhovskaya-Gorin-Sodin25}. To this end, first by Remark \ref{rmk:alt_G_beta}, we use the notation
\begin{align*}
x_j:= {}& a_{\beta,j}, &
w_j:= {}& \beta^{-1}\xi_{\beta,j}^2, & \gamma:= {}& \mathcal G_1^{(\beta),**, \Re}+\mathcal G_2^{(\beta),**, \Re},
\end{align*}
where $\{ a_{\beta, j} \}^{\infty}_{j = 1}$ are the Airy\(_\beta\) point process defined in \eqref{airy process}, $\xi_{\beta,j}^2$ are independent $\chi^2_{\beta}$ random variables in \eqref{eq:G_beta_N}, and $\mathcal G_1^{(\beta),**, \Re}$ and $\mathcal G_2^{(\beta),**, \Re}$ are defined in \eqref{eq:G_beta1_first} and \eqref{eq:G_beta2_firstc} respectively, we have
\[
\mathcal G^{(\beta)}(w)
=
\gamma+\sum_{j=1}^\infty
w_j\left(\frac{1}{w-x_j}+\frac{x_j}{1+x_j^2}\right).
\]
Hence, \(\mathcal G^{(\beta)}\in\Omega_-\) almost surely. We note that the condition \cite[Equation (8.32)]{Bykhovskaya-Gorin-Sodin25} is due to \eqref{eq:conv_G_2Im}. On the other hand, we  set
\begin{align*}
x_{j;N}:=
\begin{cases}
N^{2/3}(\mu_j^{(\beta)}-2), & j\le N,\\
N^{2/3}(\mu_N^{(\beta)}-2)-(j-N), & j>N.
\end{cases}
&
w_{j;N}:= {}&
\begin{cases}
\beta^{-1}\xi_{\beta,N,j}^2, & j\le N,\\
0, & j>N,
\end{cases}
\end{align*}
where $\{ \mu^{(\beta)}_j \}^N_{j = 1}$ are defined in \eqref{eq:defn_mu^beta_k} and $\{ \xi^2_{\beta, N, j} \}^N_{j = 1}$ are independent $\chi^2_{\beta}$ random variables in \eqref{eq:G_beta_N}, and then let
\[
\gamma_N
:=
-\sum_{j=1}^N
w_{j; N}
\frac{x_{j;N}}{1+x_{j;N}^2}
-N^{1/3}.
\]
Then
\[
\mathcal G_N^{(\beta)}(w)
=
\gamma_N+\sum_{j=1}^\infty
w_{j;N}\left(\frac{1}{w-x_{j;N}}+\frac{x_{j;N}}{1+x_{j;N}^2}\right),
\]
and it is obvious that \(\mathcal G_N^{(\beta)} \in \Omega_-\).

Second, we verify the convergence of poles and weights that is required in \cite[Equation (8.34)]{Bykhovskaya-Gorin-Sodin25}. By Proposition~\ref{pro:convergence_rank_1}, for every fixed \(M\),
\[
\bigl(x_{1;N},\dots,x_{M;N}\bigr)
=
\bigl(
N^{2/3}(\mu_1^{(\beta)}-2),\dots,N^{2/3}(\mu_M^{(\beta)}-2)
\bigr)
\Longrightarrow
(a_{\beta,1},\dots,a_{\beta,M}).
\]
For the weights, we may realize all \(\xi_{\beta,N,j}\) and \(\xi_{\beta,j}\) on one probability space so that
\[
\xi_{\beta,N,j}=\xi_{\beta,j},
\qquad 1\le j\le N.
\]
Then \(w_{j;N}=w_j\) for every fixed \(j\) and all large \(N\).Since the weights are independent of the eigenvalues, the above
convergences hold jointly in finite-dimensional distributions.

Third, we verify the asymptotics along the imaginary axis for the limit that are required in \cite[Equation (8.35)]{Bykhovskaya-Gorin-Sodin25}. Set
\[
\phi(R):=-\frac{1+i}{\sqrt2}\sqrt R.
\]
It suffices to prove that
\begin{align}
  \mathcal G^{(\beta)}(iR)-\phi(R)\to {}& 0, && \text{almost surely as }R\to\infty, \label{eq:conv_G-phi} \\
  \lim_{R\to\infty}\limsup_{N\to\infty} \mathbb E\left| \mathcal G_N^{(\beta)}(iR)-\phi(R) \right|^2 = {}& 0. && \label{eq:conv_G_N-phi}
\end{align}
The convergence \eqref{eq:conv_G-phi} is a direct consequence of Lemma \ref{lem:beta_airy_green_infty}, and we prove \eqref{eq:conv_G_N-phi} below. We recall that the Stieltjes transforms of the empirical spectral measure of $H^{(\beta)}_N$ and the semicircle law (see \eqref{Stieltjes transform of sc}) are 
\begin{align*}
m_N(z):= {}& \frac1N\sum_{j=1}^N\frac{1}{z-\mu_j^{(\beta)}}, & m(z):= {}& \int\frac{1}{z-x}\,d\mu_{\SC}(x) = \frac{z-\sqrt{z^2-4}}{2}.
\end{align*}
Then, as $\eta \in (0, +\infty)$ approaches \(0\),
\begin{equation}\label{eq:supp_cc}
m(2+i\eta)
=
1-\frac{1+i}{\sqrt2}\sqrt\eta+o(\sqrt\eta).
\end{equation}
Moreover, by
\cite[Theorem~1.1 and Remark~1.3]{Bourgade-Mody-Pain22},
there exist constants \(C,\eta_0>0\) such that
\begin{equation}\label{eq:supp_dd}
\mathbb E\left[
  |m_N(2+i\eta)-m(2+i\eta)|^2
\right]
\le \frac{C}{N^2\eta^2},
\qquad 0<\eta\le\eta_0.
\end{equation}
Indeed, analogous to \cite[Equation (8.51)]{Bykhovskaya-Gorin-Sodin25},
\[
\mathcal G_N^{(\beta)}(iR)
=
N^{1/3}\left(m_N\!\left(2+\frac{iR}{N^{2/3}}\right)-1\right)
+
\sum_{j=1}^N
\frac{\beta^{-1}\xi_{\beta,N,j}^2-1}{iR-x_{j;N}}.
\]
Taking the expectation with respect to $\xi^2_{\beta, N, j}$ first, and using that $\{ \beta^{-1} \xi^2_{\beta, N, j} - 1\}^N_{j = 1}$ are independent and identically distributed random variables with zero mean and variance $2/\beta$ and are independent of $\{ x_{j; N} \}^N_{j = 1}$, we obtain, analogous to \cite[Equation (8.51)]{Bykhovskaya-Gorin-Sodin25} 
\begin{align*}
\mathbb E\left|
\mathcal G_N^{(\beta)}(iR)-\phi(R)
\right|^2
&\le
2\mathbb E\left|
N^{1/3}\left(
m_N\!\left(2+\frac{iR}{N^{2/3}}\right)-1
\right)-\phi(R)
\right|^2
+\frac{4}{\beta}\,
\mathbb E\sum_{j=1}^N\frac{1}{R^2+x_{j;N}^2}.
\end{align*}
Applying \eqref{eq:supp_dd} with
\(\eta=RN^{-2/3}\), and then using \eqref{eq:supp_cc}, we obtain
\[
\limsup_{N\to\infty}
\mathbb E\left|
N^{1/3}\left(
m_N\!\left(2+\frac{iR}{N^{2/3}}\right)-1
\right)-\phi(R)
\right|^2
\le \frac{C}{R^2}.
\]
Hence the first term vanishes as \(R\to\infty\). For
the second term, the resolvent identity gives
\[
\sum_{j=1}^N\frac{1}{R^2+x_{j;N}^2}
=
\frac{N^{1/3}}{2iR}
\left[
m_N\!\left(2-\frac{iR}{N^{2/3}}\right)
-
m_N\!\left(2+\frac{iR}{N^{2/3}}\right)
\right]
= -\frac{N^{1/3}}{R} \Im m_N\!\left(2+\frac{iR}{N^{2/3}} \right),
\]
Applying \eqref{eq:supp_dd} again with
\(\eta=RN^{-2/3}\), and using \eqref{eq:supp_cc}, we obtain
\[
\limsup_{N\to\infty}
\mathbb E\sum_{j=1}^N\frac{1}{R^2+x_{j;N}^2}
\le
\frac{1}{\sqrt{2R}}+\frac{C}{R^2}.
\]
which tends to zero as \(R\to\infty\).

We have thus verified all the hypotheses of \cite[Theorem~8.13]{Bykhovskaya-Gorin-Sodin25} for $\mathcal{G}^{(\beta)}_N$ and $\mathcal{G}^{(\beta)}$, and complete the proof of Proposition~\ref{enu:prop:beta_airy_green_2}.

\subsection[
  Proof of Proposition and Lemma for G-functions
]{
  Proof of Proposition~\ref{pro:multispike_airy_green}
  and Lemma~\ref{lem:beta_airy_green_infty} for
  \texorpdfstring{
    $\mathcal G_{\Theta}^{(1)}$ and $\mathcal G_{\Theta}^{(2)}$
  }{
    G(1)-Theta and G(2)-Theta
  }
}
\label{eq:proof_multiple}

In this subsection, we prove Proposition~\ref{pro:multispike_airy_green}
and the part of Lemma~\ref{lem:beta_airy_green_infty} concerning
\(\mathcal G_{\Theta}^{(1)}\) and \(\mathcal G_{\Theta}^{(2)}\).
The proofs are similar to those of
\cite[Theorem~8.1 and Proposition~8.3]
{Bykhovskaya-Gorin-Sodin25}; it suffices to replace the Airy points
there by the multivariate stochastic Airy operator edge points
\(a_{\beta,j}^{(\Theta)}\). This follows from the rigidity estimate below.

We prove the case \(\beta=1\); the case \(\beta=2\) is parallel.
For notational simplicity, we write
\begin{align*}
a_j&:=a_{1,j},&
a_j^{\Theta}&:=a_{1,j}^{\Theta},&
\mathcal G_{\Theta}&:=\mathcal G_{\Theta}^{(1)}.
\end{align*}
When \(l=0\), we use the convention
\(a_j^{\Theta}=a_j^{()}=a_j\).

\begin{lem}\label{thm:robin_airy_locations}
Let $l \geq 0$ and $\Theta = (\theta_1, \dots, \theta_l)$. For every \(\varepsilon>0\), there exists a random variable \(\mathcal J' =\mathcal J'(\varepsilon)\) such that almost surely,
\begin{equation} \label{eq:rigidity_a_j^Theta}
\left|
a_j^{\Theta}
+\left(\frac{3\pi j}{2}\right)^{2/3}
\right|
\le j^\varepsilon,
\qquad j>\mathcal J'.
\end{equation}
\end{lem}

\begin{proof}
  If \(l=0\), this is \cite[Lemma~8.9]{Bykhovskaya-Gorin-Sodin25}. We assume \(l>0\).

  We use finite-\(N\) interlacing together with joint convergence. Let \(H_N\) be the GOE matrix, let \(\mu_1\ge\cdots\ge\mu_N\) be its eigenvalues, and let \(\lambda_1\ge\cdots\ge\lambda_N\) be the eigenvalues of
  \[
    H_N+\diag(d_1,\dots,d_l,0,\dots,0),
    \qquad
    d_i=1-\theta_iN^{-1/3}.
  \]
  Since \(d_i>0\) for all sufficiently large \(N\), the rank-\(l\) interlacing inequalities give
  \[
    \mu_j\le \lambda_j\le \mu_{j-l},
    \qquad j>l,
  \]
  with the convention that \(\mu_m=+\infty\) for \(m\le0\). Equivalently, at the soft edge,
  \[
    N^{2/3}(\mu_j-2)
    \le
    N^{2/3}(\lambda_j-2)
    \le
    N^{2/3}(\mu_{j-l}-2),
    \qquad j>l.
  \]
  By the soft-edge convergence of the GOE and by Proposition~\ref{pro:convergence_rank_r}, for each fixed \(M\)
the two marginal vectors
\[
\bigl(N^{2/3}(\mu_1-2),\ldots,N^{2/3}(\mu_M-2)\bigr),
\qquad
\bigl(N^{2/3}(\lambda_1-2),\ldots,N^{2/3}(\lambda_M-2)\bigr)
\]
are tight and converge respectively to
\[
(a_1,\ldots,a_M),\qquad (a^\Theta_1,\ldots,a^\Theta_M).
\]
Hence, after passing to a joint subsequential limit, we obtain a coupling of the two limiting
point processes. Since the finite-\(N\) interlacing inequalities define a closed condition, they pass
to the limit, yielding
\[
a_j\le a^\Theta_j\le a_{j-l},\qquad l<j\le M.
\]
A diagonal argument in \(M\) gives this coupling for all \(j>l\).
  
  Applying
\cite[Lemma~8.9]{Bykhovskaya-Gorin-Sodin25} with \(\varepsilon/2\), we have almost surely, for all large \(j\),
\[
\left|a_j+\left(\frac{3\pi j}{2}\right)^{2/3}\right|
+
\left|a_{j-l}+\left(\frac{3\pi(j-l)}{2}\right)^{2/3}\right|
\le 2j^{\varepsilon/2}.
\]
Since \(l\) is fixed,
\[
\left(\frac{3\pi(j-l)}2\right)^{2/3}
=
\left(\frac{3\pi j}2\right)^{2/3}+O(j^{-1/3}).
\]
Together with
\[
a_j\le a^\Theta_j\le a_{j-l},
\]
this gives
\[
\left|a^\Theta_j+\left(\frac{3\pi j}{2}\right)^{2/3}\right|
\le j^\varepsilon
\]
for all sufficiently large \(j\), after increasing the random threshold.
\end{proof}
\begin{proof}[Proof of Proposition~\ref{pro:multispike_airy_green}]
  The proof is the same as the proof of \cite[Theorem~8.1]{Bykhovskaya-Gorin-Sodin25}. Indeed, in that proof the only input on the pole locations is the rigidity estimate for the Airy points, and the only input on the weights is that, after subtracting one, they are independent centered random variables with finite variance. By Lemma~\ref{thm:robin_airy_locations}, the same rigidity estimate holds for \(a_{j}^{\Theta}\). Therefore the argument in \cite[Theorem~8.1]{Bykhovskaya-Gorin-Sodin25} applies verbatim after replacing \(a_j\) by \(a_{j}^{\Theta}\). This proves the existence of the limit in \eqref{eq:defn_Airy_Green_multiple}, uniformly on compact subsets avoiding the poles.
\end{proof}
\begin{proof}[Proof of Lemma~\ref{lem:beta_airy_green_infty}
for \(\mathcal G_{\Theta}^{(\beta)}\), \(\beta\in\{1,2\}\)]
We follow the proof of
\cite[Proposition~8.3]{Bykhovskaya-Gorin-Sodin25}.
By Lemma \ref{thm:robin_airy_locations} for \(\beta=1\), and by its
parallel GUE version for \(\beta=2\), for every \(\varepsilon>0\),
almost surely,
\[
a_{\beta,j}^{\Theta}
=
-\left(\frac{3\pi j}{2}\right)^{2/3}
+O(j^\varepsilon),
\qquad j\to\infty.
\]
Moreover, the exponential tail of the chi-square distribution and
the Borel--Cantelli lemma give
\[
\beta^{-1}\xi_{\beta,j}^{2}=O(j^\varepsilon)
\qquad\text{almost surely}.
\]

Using
\[
\frac{\beta^{-1}\xi_{\beta,j}^{2}}
{w-a_{\beta,j}^{\Theta}}
=
\frac1{w+\left(\frac{3\pi j}{2}\right)^{2/3}}
+
\frac{\beta^{-1}\xi_{\beta,j}^{2}-1}
{w+\left(\frac{3\pi j}{2}\right)^{2/3}}
+
\beta^{-1}\xi_{\beta,j}^{2}
\frac{
\left(\frac{3\pi j}{2}\right)^{2/3}
+a_{\beta,j}^{\Theta}
}{
\bigl(w-a_{\beta,j}^{\Theta}\bigr)
\left(w+\left(\frac{3\pi j}{2}\right)^{2/3}\right)
},
\]
we obtain, as in
\cite[Equation~(8.25)]{Bykhovskaya-Gorin-Sodin25},
the decomposition
\begin{align}
\mathcal G_{\Theta}^{(\beta)}(w)
={}&
\lim_{y\to-\infty}
\left[
\sum_{\left(\frac{3\pi j}{2}\right)^{2/3}<-y}
\frac1{w+\left(\frac{3\pi j}{2}\right)^{2/3}}
-\frac2\pi\sqrt{-y}
\right]                                                    \notag\\
&+
\sum_{j=1}^{\infty}
\frac{\beta^{-1}\xi_{\beta,j}^{2}-1}
{w+\left(\frac{3\pi j}{2}\right)^{2/3}}                    \notag\\
&+
\sum_{j=1}^{\infty}
\beta^{-1}\xi_{\beta,j}^{2}
\frac{
\left(\frac{3\pi j}{2}\right)^{2/3}
+a_{\beta,j}^{\Theta}
}{
\bigl(w-a_{\beta,j}^{\Theta}\bigr)
\left(w+\left(\frac{3\pi j}{2}\right)^{2/3}\right)
}.
\label{eq:multispike-Gorin-decomposition}
\end{align}
Here the rigidity estimate above plays the same role as
\cite[Lemma~8.9]{Bykhovskaya-Gorin-Sodin25} in replacing the cutoff
\(a_{\beta,j}^{\Theta}>y\) by
\(\left(\frac{3\pi j}{2}\right)^{2/3}<-y\).

As in \cite[Equations~(8.28)--(8.29)]
{Bykhovskaya-Gorin-Sodin25}, comparison with the corresponding
integral gives, uniformly for \(\Re w\ge0\),
\[
\lim_{y\to-\infty}
\left[
\sum_{\left(\frac{3\pi j}{2}\right)^{2/3}<-y}
\frac1{w+\left(\frac{3\pi j}{2}\right)^{2/3}}
-\frac2\pi\sqrt{-y}
\right]
=
-\sqrt w+O(|w|^{-1/2}).
\]

Next, the law of the iterated logarithm gives
\[
\sum_{k=1}^{j}
\bigl(\beta^{-1}\xi_{\beta,k}^{2}-1\bigr)
=
O\bigl(\sqrt{j\log\log(j+2)}\bigr)
\qquad\text{almost surely}.
\]
The second term in
\eqref{eq:multispike-Gorin-decomposition} is treated by the
summation-by-parts argument used for the second sum in
\cite[Equation~(8.25), in the paragraph between Equations~(8.27)
and~(8.28)]{Bykhovskaya-Gorin-Sodin25}. It follows that
\[
\sum_{j=1}^{\infty}
\frac{\beta^{-1}\xi_{\beta,j}^{2}-1}
{w+\left(\frac{3\pi j}{2}\right)^{2/3}}
\longrightarrow0
\]
almost surely, uniformly as \(|w|\to\infty\) with \(\Re w\ge0\).

For the last term in
\eqref{eq:multispike-Gorin-decomposition}, we use the estimate
\cite[Equation~(8.26)]{Bykhovskaya-Gorin-Sodin25} and the
three-range decomposition in
\cite[Equation~(8.27)]{Bykhovskaya-Gorin-Sodin25}, splitting the
sum at \(j=\lfloor |w|^{3/2}\rfloor\). For
\(0<\varepsilon<1/12\), the rigidity estimate and the chi-square
tail bound imply
\begin{align*}
  \left|
\sum_{j=1}^{\infty}
\beta^{-1}\xi_{\beta,j}^{2}
\frac{
\left(\frac{3\pi j}{2}\right)^{2/3}
+a_{\beta,j}^{\Theta}
}{
\bigl(w-a_{\beta,j}^{\Theta}\bigr)
\left(w+\left(\frac{3\pi j}{2}\right)^{2/3}\right)
}
\right| %\\
%&\qquad
  \le {}&
O(|w|^{-2})
+C|w|^{-2}\sum_{j\le |w|^{3/2}}j^{2\varepsilon}
+C\sum_{j>|w|^{3/2}}j^{-4/3+2\varepsilon}\\
= {}&
O(|w|^{-1/2+3\varepsilon})
=o(|w|^{-1/4}).
\end{align*}
The same calculation after \(m\) differentiations gives
\[
\frac{d^m}{dw^m}
\sum_{j=1}^{\infty}
\beta^{-1}\xi_{\beta,j}^{2}
\frac{
\left(\frac{3\pi j}{2}\right)^{2/3}
+a_{\beta,j}^{\Theta}
}{
\bigl(w-a_{\beta,j}^{\Theta}\bigr)
\left(w+\left(\frac{3\pi j}{2}\right)^{2/3}\right)
}
=
O(|w|^{-m-1/2+3\varepsilon}),
\qquad m=1,2.
\]

We now prove the joint fluctuation statement directly from
\eqref{eq:multispike-Gorin-decomposition}. Termwise
differentiation is justified by local uniform convergence away from
the poles. The estimates above and their differentiated versions give
\begin{align*}
%&
\begin{pmatrix}
x^{1/4}\bigl(\mathcal G_{\Theta}^{(\beta)}(x)+\sqrt x\bigr)\\[0.4em]
x^{5/4}\left(
\bigl(\mathcal G_{\Theta}^{(\beta)}\bigr)'(x)
+\dfrac1{2\sqrt x}
\right)
\end{pmatrix} %\\
 %&\qquad
   =
\sum_{j=1}^{\infty}
\bigl(\beta^{-1}\xi_{\beta,j}^{2}-1\bigr)
\begin{pmatrix}
\dfrac{x^{1/4}}
{x+\left(\frac{3\pi j}{2}\right)^{2/3}}\\[1em]
-\dfrac{x^{5/4}}
{\left(x+\left(\frac{3\pi j}{2}\right)^{2/3}\right)^2}
\end{pmatrix}
+o_{\mathbb P}(1).
\end{align*}
For every \(p>3/2\), the same sum--integral comparison as in
\cite[Equation~(8.30)]{Bykhovskaya-Gorin-Sodin25} gives
\[
x^{p-3/2}
\sum_{j=1}^{\infty}
\frac1{
\left(x+\left(\frac{3\pi j}{2}\right)^{2/3}\right)^p
}
\longrightarrow
\frac1\pi\int_0^\infty
\frac{u^{1/2}}{(1+u)^p}\,du.
\]
The case \(p=2\) is precisely the variance calculation in
\cite[Equation~(8.30)]{Bykhovskaya-Gorin-Sodin25}; the cases
\(p=3,4\) give the covariance and the variance of the derivative
component. Since
\[
\operatorname{Var}
\bigl(\beta^{-1}\xi_{\beta,j}^{2}-1\bigr)
=\frac2\beta,
\]
the covariance matrix of the centered sum converges to
\[
\frac1\beta
\begin{pmatrix}
1&-\frac14\\[0.2em]
-\frac14&\frac18
\end{pmatrix}.
\]
The largest coefficient in the triangular array is
\(O(x^{-3/4})\), while
\[
\mathbb E
\left|
\beta^{-1}\xi_{\beta,j}^{2}-1
\right|^3<\infty.
\]
Thus the Lyapunov condition and the Cramér--Wold device give
\[
\begin{pmatrix}
x^{1/4}\bigl(\mathcal G_{\Theta}^{(\beta)}(x)+\sqrt x\bigr)\\[0.4em]
x^{5/4}\left(
\bigl(\mathcal G_{\Theta}^{(\beta)}\bigr)'(x)
+\dfrac1{2\sqrt x}
\right)
\end{pmatrix}
\Longrightarrow
\mathcal N\left(
\begin{pmatrix}0\\0\end{pmatrix},
\frac1\beta
\begin{pmatrix}
1&-\frac14\\[0.2em]
-\frac14&\frac18
\end{pmatrix}
\right).
\]

Since the second component is tight,
\[
x^{1/4}
\left[
-\frac1{
\bigl(\mathcal G_{\Theta}^{(\beta)}\bigr)'(x)}
-2\sqrt x
\right]
=
4x^{5/4}
\left(
\bigl(\mathcal G_{\Theta}^{(\beta)}\bigr)'(x)
+\frac1{2\sqrt x}
\right)
+o_{\mathbb P}(1).
\]
It follows that
\[
\begin{pmatrix}
x^{1/4}\bigl(\mathcal G_{\Theta}^{(\beta)}(x)+\sqrt x\bigr)\\[0.4em]
x^{1/4}\left[
-\dfrac1{
\bigl(\mathcal G_{\Theta}^{(\beta)}\bigr)'(x)}
-2\sqrt x
\right]
\end{pmatrix}
\Longrightarrow
\mathcal N\left(
\begin{pmatrix}0\\0\end{pmatrix},
\frac1\beta
\begin{pmatrix}
1&-1\\
-1&2
\end{pmatrix}
\right).
\]

Finally, the three estimates following
\eqref{eq:multispike-Gorin-decomposition} imply that, almost surely,
\[
\mathcal G_{\Theta}^{(\beta)}(w)+\sqrt w\longrightarrow0
\]
uniformly as \(|w|\to\infty\), \(\Re w\ge0\). Applying Cauchy's
formula to
\(\mathcal G_{\Theta}^{(\beta)}(w)+\sqrt w\) on
\(|w-x|=x/2\) gives
\[
\bigl(\mathcal G_{\Theta}^{(\beta)}\bigr)'(x)
=
-\frac1{2\sqrt x}+o(x^{-1}),
\qquad
\bigl(\mathcal G_{\Theta}^{(\beta)}\bigr)''(x)
=
\frac1{4x^{3/2}}+o(x^{-2}),
\]
which completes the proof.
\end{proof}
\subsection{Proof of Proposition \ref{prop:general-meromorphic}} \label{subsec:proof_prop:general-meromorphic}

The proof is based on \cite[Theorem~8.13]{Bykhovskaya-Gorin-Sodin25}, the same as the proof of Proposition~\ref{enu:prop:beta_airy_green_2} and the proof of \cite[Theorem 8.20]{Bykhovskaya-Gorin-Sodin25}. We need the following lemma whose proof will be given in Section \ref{sec:linear_algebra}.

\begin{lem}[Spherical second and fourth moments in a subspace]
  \label{lem:spherical-moments-subspace}
  Let \(\mathbb F\in\{\mathbb R,\mathbb C\}\), and put
  \[
    \beta_{\mathbb F}:=\dim_{\mathbb R}\mathbb F
    =
    \begin{cases}
      1, & \mathbb F=\mathbb R,\\
      2, & \mathbb F=\mathbb C.
    \end{cases}
  \]
  Equip \(\mathbb F^N\) with the standard Euclidean/Hermitian inner product. Let \(V\subset \mathbb F^N\) be an \(m\)-dimensional \(\mathbb F\)-linear subspace, and let \(P\) be the matrix representing the orthogonal projection onto \(V\). If \(v\) is a random unit vector in \(V\) uniformly distributed on the unit sphere of \(V\), then, with $v^\ast$ understood as $v^{\top}$ if $\mathbb F = \realR$,
  \begin{equation}\label{eq:vvtop}
    \mathbb E[vv^\ast]=\frac1m P.
  \end{equation}
  Moreover, for any \(N\times N\) matrices \(A\) and \(B\) that are real symmetric if \(\mathbb F = \realR\) and are Hermitian if \(\mathbb F = \compC\),
  \begin{equation}\label{eq:trace_identity}
    \mathbb E\bigl[(v^\ast A v)(v^\ast B v)\bigr]
    =
    \frac{
      \tr(PAP)\tr(PBP)+\frac{2}{\beta_{\mathbb F}}\tr(PAPBP)
    }{
      m\bigl(m+\frac{2}{\beta_{\mathbb F}}\bigr)
    }.
  \end{equation}
\end{lem}

\begin{proof}[Proof of Proposition \ref{prop:general-meromorphic}]
  For notational simplicity, we adapt the short-handed notation as in \eqref{eq:notation_mult_GOE} in Section \ref{subsec:proof_mult_rank_GOE}. We also make use of notation in Section \ref{eq:proof_multiple}.
  
First, we verify that $\widetilde{\mathcal{G}}_{N, r}(w)$ and $\mathcal G_{(r - 1), r^c}(w)$ are all in the function space $\Omega_-$ defined in \cite[Definition 8.10]{Bykhovskaya-Gorin-Sodin25}. To this end, we denote
\begin{align*}
x_j:= {}& a_j^{\Theta^{(r)}}, & w_j:= {}& \xi_j^2, & \gamma = \mathcal{G}^{**, \Re}_{(r - 1), r^c, 1} + \mathcal{G}^{**, \Re}_{(r - 1), r^c, 2},
\end{align*}
where \((\xi_j)_{j\ge1}\) are i.i.d.\ standard Gaussian random variables, independent of the limiting edge eigenvalue process, and
\begin{align*}
  \mathcal{G}^{**, \Re}_{(r - 1), r^c, 1} = {}& \lim_{x\to-\infty} \Biggl[ \sum_{j:\,x_j>x}\frac{-x_j}{1+x_j^2} -\frac{2}{\pi}\sqrt{-x} \Biggr], & \mathcal{G}^{**, \Re}_{(r - 1), r^c, 2} = {}& \lim_{x\to-\infty} \sum_{j:\,x_j>x} (\xi^2_j - 1) \frac{-x_j}{1+x_j^2}.
\end{align*}
Since $a_j^{\Theta^{(r)}}$ satisfies inequality \eqref{eq:rigidity_a_j^Theta} by Lemma~\ref{thm:robin_airy_locations}, which is analogous to the inequality \eqref{eq:est_ap_k} satisfied by $a_{\beta, j}$, by arguments in the proof of  Proposition~\ref{prop:beta_airy_green}, especially the well-definedness of $\mathcal G_1^{(\beta),**, \Re}$ and $\mathcal G_2^{(\beta),**, \Re}$, we have that $\mathcal{G}^{**, \Re}_{(r - 1), r^c, 1}$ and $\mathcal{G}^{**, \Re}_{(r - 1), r^c, 2}$ are well-defined random variables. As in the corresponding \(\Omega_-\) representation for \(\mathcal G^{(\beta)}\), we have
\begin{equation*} % \label{eq:Omega-minus-representation-G-r}
  \mathcal G_{(r - 1), r^c}(w)
  =
  \gamma+\sum_{j=1}^\infty
  w_j\left(\frac{1}{w-x_j}+\frac{x_j}{1+x_j^2}\right).
\end{equation*}
Thus \(\mathcal G_{(r-1),r^c}\in\Omega_-\) almost surely. On the other hand, we  set
\begin{align*}
x_{j;N}:=
\begin{cases}
a_{j,N}, & j\le N,\\
a_{N,N}-(j-N), & j>N.
\end{cases}
&
w_{j;N}:= {}&
\begin{cases}
N|\langle e_r,\widehat{\u}_{j,N}\rangle|^2, & j\le N,\\
0, & j>N,
\end{cases}
\end{align*}
where $\{ a_{j,N} \}^N_{j = 1}$ and $\{ N|\langle e_r,\widehat{\u}_{j,N}\rangle|^2 \}^N_{j = 1}$ are defined in \eqref{eq:defn_Gtilde_N_r}, and then let
\[
\gamma_N
:=
-\sum_{j=1}^N
w_{j; N}
\frac{x_{j;N}}{1+x_{j;N}^2}
-N^{1/3}.
\]
Then
\[
\widetilde{\mathcal G}_{N,r}(w)
=
\gamma_N+\sum_{j=1}^\infty
w_{j;N}\left(\frac{1}{w-x_{j;N}}+\frac{x_{j;N}}{1+x_{j;N}^2}\right),
\]
and it is obvious that \(\widetilde{\mathcal G}_{N,r} \in \Omega_-\).

Second, we verify the convergence of poles and weights that is required in \cite[Equation (8.34)]{Bykhovskaya-Gorin-Sodin25}. This follows from Lemma~\ref{lem:joint-edge-gaussian-general-short}.

Set
\begin{equation}\label{eq:def_phi_R}
\phi(R):=-\frac{1+i}{\sqrt2}\sqrt R .
\end{equation}

Third, we verify the asymptotics along the imaginary axis for the limit that are required in \cite[Equation (8.35)]{Bykhovskaya-Gorin-Sodin25}. With \(\phi(R)\) defined in \eqref{eq:def_phi_R}, by Lemma~\ref{lem:beta_airy_green_infty} 
\[
\mathcal G_{(r-1),r^c}(iR)-\phi(R)\to0
\qquad\text{almost surely as }R\to\infty.
\]
It remains to show that
\begin{equation} \label{eq:mult_GOE_third}
\lim_{R\to\infty}\limsup_{N\to\infty}
\mathbb E\left|
\widetilde{\mathcal G}_{N,r}(iR)-\phi(R)
\right|^2
=0.
\end{equation}

For the GOE matrix \(H_N\), set
\[
m_N(z):=\frac1N\sum_{j=1}^N\frac{1}{z-\mu_j},
\qquad
m(z):=G_{\SC}(z)=\frac{z-\sqrt{z^2-4}}{2}.
\]
As \(\eta\downarrow0\),
\begin{equation}\label{eq:supp_cc1}
m(2+i\eta)
=
1-\frac{1+i}{\sqrt2}\sqrt\eta+o(\sqrt\eta).
\end{equation}
Moreover, by the local \(L^2\) estimate for the Gaussian \(\beta\)-ensemble \cite{Bourgade-Mody-Pain22}, for every \(\varepsilon>0\), there is a constant \(C=C(\varepsilon)>0\) such that
\begin{equation}\label{eq:supp_dd1}
\mathbb E\!\left[
|m_N(2+i\eta)-m(2+i\eta)|^2
\right]
\le \frac{C}{N^2\eta^2},
\qquad
N^{-2/3}<\eta<N^{-2/3+\varepsilon}.
\end{equation} 
We define
\begin{equation}\label{eq:defn_mhat}
\widehat{m}_N(z)
=
\frac1N \sum_{j=1}^N \frac{1}{z-\widehat{\mu}_{j,N}} .
\end{equation}
where $\{\widehat{\mu}_{j,N}: 1 \le j \le N\}$ are the eigenvalues of
\begin{equation}\label{eq:hat_H_N}
\widehat H_N
:=
H_N+\sum_{j=1}^{r-1} d_j e_j e_j^\top .
\end{equation}
By the classical interlacing property for eigenvalues under a rank-$1$ (or finite-rank) perturbation, we have
\[
\widehat{\mu}_{j,N}  \ge \mu_j \ge \widehat{\mu}_{j + r - 1,N}, \qquad 1 \le j \le N - r + 1,
\]
where $\{\mu_{j}\}_{j=1}^N$ are the eigenvalues of $H_N^{}$. 
By the rank inequality for Stieltjes transforms, for every
\(\eta>0\),
\[
\left|
\widehat m_N(2+i\eta)-m_N(2+i\eta)
\right|
\le \frac{C_r}{N\eta}.
\]
Combining this bound with \eqref{eq:supp_dd1}, we obtain
\begin{equation}\label{eq:est_mhat}
\mathbb E\left[
\left|
\widehat m_N(2+i\eta)-m(2+i\eta)
\right|^2
\right]
\le
\frac{C_r}{N^2\eta^2},
\qquad
N^{-2/3}<\eta<N^{-2/3+\varepsilon}.
\end{equation}
We define the vectors 
\[
\alpha_{j,N}
:=
\bigl(
\langle e_j,\widehat{\u}_{1,N}\rangle,\dots,
\langle e_j,\widehat{\u}_{N,N}\rangle
\bigr)^{\top} \in\mathbb R^N, \quad j = 1, \dotsc, N,
\]
and set
\[
\mathcal F_{r-1,N}
:=
\sigma\!\left(
\widehat\mu_{1,N},\dots,\widehat\mu_{N,N},
\alpha_{1,N},\dots,\alpha_{r-1,N}
\right),
\]
where \(\widehat\mu_{1,N},\dots,\widehat\mu_{N,N}\) are the eigenvalues of \(\widehat H_N\) defined in \eqref{eq:hat_H_N}, and \(\alpha_{j,N}\) is defined above. Since the vectors $\alpha_{j, N}$ are orthonormal, the operator $P_N$ defined by the matrix
\[
P_{r - 1, N}:=I-\sum_{j=1}^{r-1}\alpha_{j,N}\alpha_{j,N}^\top
\]
is the orthogonal projection onto
\[
\mathcal V_{r - 1, N}
:=
\Bigl(\Span\{\alpha_{1,N},\dots,\alpha_{r-1,N}\}\Bigr)^\perp.
\]
By the residual orthogonal-invariance argument in the proof of
Lemma~\ref{lem:joint-edge-gaussian-general-short}, conditionally on
\(\mathcal F_{r-1,N}\), the vector \(\alpha_{r,N}\), modulo the
irrelevant signs of the eigenvectors, is uniformly distributed on
the unit sphere of \(\mathcal V_{r-1,N}\), whose dimension is
\(N-r+1\).

Fix \(R>0\), and define 
\[
d_{j,N}(R):=iR-a_{j,N},
\qquad
D_N(R):=\diag\!\left(\frac1{d_{1,N}(R)},\dots,\frac1{d_{N,N}(R)}\right).
\]
where \(a_{j,N}:=N^{2/3}(\widehat\mu_{j,N}-2)\), as in
\eqref{eq:defn_Gtilde_N_r}. Then
\begin{equation} \label{eq:G(iR)_in_v}
\widetilde{\mathcal G}_{N,r}(iR)+N^{1/3}
=
N\, \alpha_{r, N}^\top D_N(R) \alpha_{r, N}.
\end{equation}

Since \(D_N(R)\) is diagonal, the signs of the eigenvectors do not
affect the quadratic form below. Hence, by \eqref{eq:vvtop} in
Lemma~\ref{lem:spherical-moments-subspace}, conditionally on
\(\mathcal F_{r-1,N}\),
\[
\mathbb E\!\left[
\alpha_{r,N}^\top D_N(R)\alpha_{r,N}
\mid\mathcal F_{r-1,N}
\right]
=
\frac{1}{N-r+1}
\tr\bigl(P_{r-1,N}D_N(R)\bigr).
\]
Therefore, \eqref{eq:G(iR)_in_v} implies
\[
\mathbb E\!\left[
\widetilde{\mathcal G}_{N,r}(iR)
\mid\mathcal F_{r-1,N}
\right]
=
\frac{N}{N-r+1}
\tr\bigl(P_{r-1,N}D_N(R)\bigr)-N^{1/3}.
\]
We denote
\[
A_N(R):= \tr D_N(R)-N^{1/3}.
\]
Since
\[
\tr(P_{r - 1, N} D_N(R))
=
\tr D_N(R)-\tr\sum_{j=1}^{r-1}\alpha_{j,N}^\top D_N(R)\alpha_{j,N},
\]
and \(|d_{m,N}(R)|\ge R\), we get
\[
|E[\widetilde{\mathcal G}_{N,r}(iR) \mid \mathcal{F}_{r - 1, N}] -A_N(R)|
\le
\frac{r}{N}|\tr D_N(R)|+\frac{r}{R}
\]
for all large \(N\). On the other hand,
\[
A_N(R)
=
N^{1/3} \widehat{m}_N\!\left(2+\frac{iR}{N^{2/3}}\right)-N^{1/3},
\]
where \( \widehat{m}_N\) is defined in \eqref{eq:defn_mhat}. Using the estimate \eqref{eq:est_mhat} and \eqref{eq:supp_cc1}, and arguing as
in the proof of \cite[Theorem~8.20]{Bykhovskaya-Gorin-Sodin25}, we obtain
\begin{equation} \label{eq:mult_GOE_3_1}
\lim_{R\to\infty}\limsup_{N\to\infty}
\mathbb E\left|
E[\widetilde{\mathcal G}_{N,r}(iR) \mid \mathcal{F}_{r - 1, N}]
-\phi(R)
\right|^2=0.
\end{equation}

Since we have
\begin{equation} \label{eq:mult_GOE_3_2}
\mathbb E\left|
\widetilde{\mathcal G}_{N,r}(iR)-\phi(R)
\right|^2
=
\mathbb E|E[\widetilde{\mathcal G}_{N,r}(iR) \mid \mathcal{F}_{r - 1, N}]-\phi(R)|^2
+
\mathbb E\,\Var\!\left(\widetilde{\mathcal G}_{N,r}(iR)\mid \mathcal F_{r-1,N}\right).
\end{equation}
It remains to estimate the conditional variance $\Var(\widetilde{\mathcal G}_{N,r}(iR)\mid \mathcal F_{r-1,N})$. Set
\[
X_N(R):=N\,\alpha_{r, N}^\top D_N(R) \alpha_{r, N},
\qquad
n_N:=N-r+1.
\]
Since
\[
\widetilde{\mathcal G}_{N,r}(iR)=X_N(R)-N^{1/3},
\]
we have
\[
\Var\!\left(\widetilde{\mathcal G}_{N,r}(iR)\mid \mathcal F_{r-1,N}\right)
=
\Var\!\left(X_N(R)\mid \mathcal F_{r-1,N}\right).
\]

Write
\[
D_N(R)=D^{\Re}_N(R)+iD^{\Im}_N(R),
\]
where \(D^{\Re}_N(R)=\Re D_N(R)\) and \(D^{\Im}_N(R)=\Im D_N(R)\) are real diagonal, hence self-adjoint, matrices. Applying Lemma~\ref{lem:spherical-moments-subspace} separately to \(D^{\Re}_N(R)\) and \(D^{\Im}_N(R)\), we obtain
\[
\Var\!\left(X_N(R)\mid \mathcal F_{r-1,N}\right)
\le
C_r
\left(
\|P_{r-1,N}D^{\Re}_N(R)P_{r-1,N}\|_{HS}^2
+
\|P_{r-1,N}D^{\Im}_N(R)P_{r-1,N}\|_{HS}^2
\right).
\]
Since
\[
\|P_{r-1,N}D^{\Re}_N(R)P_{r-1,N}\|_{HS}^2
+
\|P_{r-1,N}D^{\Im}_N(R)P_{r-1,N}\|_{HS}^2
\le
\|D_N(R)\|_{HS}^2,
\]
we get
\[
\Var\!\left(X_N(R)\mid \mathcal F_{r-1,N}\right)
\le
C_r\sum_{m=1}^N \frac{1}{|iR-a_{m,N}|^2}.
\]
Here \(\|A\|_{HS}:=(\tr(AA^\ast))^{1/2}\). 
Hence
\[
\Var\!\left(\widetilde{\mathcal G}_{N,r}(iR)\mid \mathcal F_{r-1,N}\right)
\le
C_r\sum_{m=1}^N \frac1{|iR-a_{m,N}|^2}.
\]
Using
\[
\sum_{m=1}^N \frac1{|iR-a_{m,N}|^2}
=
\frac{N^{1/3}}{2iR}
\left[
\widehat{m}_N\!\left(2-\frac{iR}{N^{2/3}}\right)
-
\widehat{m}_N\!\left(2+\frac{iR}{N^{2/3}}\right)
\right],
\]
and the same local \(L^2\) estimate again, we conclude that
\[
\lim_{R\to\infty}\limsup_{N\to\infty}
\mathbb E\,\Var\!\left(\widetilde{\mathcal G}_{N,r}(iR)\mid \mathcal F_{r-1,N}\right)=0,
\]
and together with \eqref{eq:mult_GOE_3_1} and \eqref{eq:mult_GOE_3_2}, we verify \eqref{eq:mult_GOE_third}.

Thus all the three conditions required by \cite[Theorem~8.13]{Bykhovskaya-Gorin-Sodin25} hold, and \cite[Theorem~8.13]{Bykhovskaya-Gorin-Sodin25} yields Proposition \ref{prop:general-meromorphic}.
\end{proof}

\section{Proof of two linear algebraic results} \label{sec:linear_algebra}

Here we prove Lemma \ref{lem:eigenvector_eigenvalue} and a technical linear algebraic result used in the proof of Theorems \ref{finite_spikes_GOE} and \ref{one_spike_beta}.

\begin{proof}[Proof of Lemma \ref{lem:eigenvector_eigenvalue}]
  For notational simplicity, we only prove the lemma in the real symmetric setting. The Hermitian setting is parallel.
  
  Recall that \(\x_k\) is the normalized eigenvector of \( Y = X + d \w \w^{\top} \) associated with \(\lambda_k\). Since
\[
  (X + d \w \w^{\top}) \x_k=\lambda_k\x_k,
\]
we have
\[
  (X - \lambda_k I)\x_k=-d (\w^\top \x_k) \w .
\]
Because \(\lambda_k\notin \{ \mu_j \}^N_{j = 1} = \sigma(X)\), the matrix \(X - \lambda_k I\) is invertible, and hence \(\x_k\) is proportional to \((X - \lambda_k I)^{-1} \w \). More precisely, up to a sign, 
\begin{equation*} % \label{eq:x_k_in_v_quantitative}
  \x_k = \frac{(X - \lambda_k I)^{-1} \w}{\sqrt{\w^\top (X - \lambda_k I)^{-2} \w}}.
\end{equation*}

Since \(\lambda_k\) is an eigenvalue of \(Y\), the matrix determinant lemma \cite[Theorem 18.1.1]{Harville97} gives 
\begin{equation*} % \label{eq:det_trick_rank_one}
  \begin{split}
    0 = \det(\lambda_k I - Y) 
      &= \det(\lambda_k I - X)\det\bigl(I-d(\lambda_k I - X)^{-1} \w  \w^\top\bigr) \\
      &= \det(\lambda_k I - X)\Bigl(1-d \w^\top(\lambda_k I - X)^{-1} \w \Bigr).
  \end{split}
\end{equation*}
Since \(\det(\lambda_k I - X)\neq0\), it follows that \eqref{eq:rank_one_root_identity} holds. Hence, \eqref{eq:eigenvector_eigenvalue} is also proved.
\end{proof}

\begin{proof}[Proof of Lemma \ref{lem:spherical-moments-subspace}]
  By orthogonal/unitary invariance, we may assume
  \begin{align*}
    V= {}& \Span_{\mathbb F}\{e_1,\dots,e_m\},
    &
      P= {}& I_m\oplus 0_{N-m}.
  \end{align*}
  Let \(\zeta_1,\dots,\zeta_m\) be i.i.d.\ standard \(\mathbb F\)-Gaussian variables, real standard normal if \(\mathbb F=\mathbb R\), and standard complex normal if \(\mathbb F=\mathbb C\). Then
  \[
    v=\frac{(\zeta_1,\dots,\zeta_m,0,\dots,0)^\top}
    {\left(\sum_{i=1}^m |\zeta_i|^2\right)^{1/2}}
  \]
  is uniformly distributed on the unit sphere of \(V\). This gives that for $1\le a,b\le m$,
  \begin{align*}
    \mathbb E[v_a v_b]= {}& \frac{\delta_{ab}}{m}, \quad (\mathbb{F} = \realR), & \mathbb E[\overline v_a v_b]= {}& \frac{\delta_{ab}}{m}, \quad (\mathbb{F} = \compC),\qquad 
  \end{align*}
  and hence \eqref{eq:vvtop}.
  
  For the fourth moment, the standard spherical moment identities are
  \begin{align*}
    \mathbb E[v_a v_b v_c v_d]
    = {}&
    \frac{
      \delta_{ab}\delta_{cd}
      +\delta_{ac}\delta_{bd}
      +\delta_{ad}\delta_{bc}
    }{m(m+2)}
          &&
          (\mathbb F=\mathbb R), \\
    \mathbb E[\overline v_a v_b \overline v_c v_d]
    = {}&
    \frac{
      \delta_{ab}\delta_{cd}
      +\delta_{ad}\delta_{bc}
    }{m(m+1)}
          &&
             (\mathbb F=\mathbb C).
  \end{align*}
  Contracting these identities with the entries of the compressions \(PAP\) and	\(PBP\), and using that \(A,B\) are self-adjoint, gives
  \[
    \mathbb E\bigl[(v^\ast A v)(v^\ast B v)\bigr]
    =
    \begin{cases}
      \dfrac{\tr(PAP)\tr(PBP)+2\tr(PAPBP)}{m(m+2)},
      & \mathbb F=\mathbb R,\\[1.2em]
      \dfrac{\tr(PAP)\tr(PBP)+\tr(PAPBP)}{m(m+1)},
      & \mathbb F=\mathbb C.
    \end{cases}
  \]
  This is exactly \eqref{eq:trace_identity}.
\end{proof}

\end{document}